\documentclass[12pt]{article}
\usepackage{amsmath}
\usepackage[english]{babel}
\usepackage[T1]{fontenc}
\usepackage{epsf}
\usepackage{fancybox}
\usepackage{fancyhdr}
\usepackage{amssymb,amsfonts,amsthm}

\usepackage{enumitem}
\usepackage{pdfsync}
\usepackage{hyperref}
\usepackage{pdfpages}
\usepackage{comment}
{\begin{center}\begin{Sbox}\begin{minipage}{\linewidth}\begin{small}\sf\color{red}(#1)}%
{\end{small}\end{minipage}\end{Sbox}\end{center}\colorbox{yellow!90!black}{\TheSbox}}
{\begin{center}\begin{Sbox}\begin{minipage}{\linewidth}\begin{small}\sf\color{black}(#1)}%
{\end{small}\end{minipage}\end{Sbox}\end{center}\colorbox{green!90!black}{\TheSbox}}
\newtheorem{lemma}{Lemma}
\newtheorem{conjecture}{Conjecture}
\newtheorem{lemma*}{Lemma}
\newtheorem{property}{Property}
\newtheorem{theorem}{Theorem}

\newtheorem{remark}{Remark}
\newtheorem{example}{Example}
\newtheorem{proposition}{Proposition}
\theoremstyle{definition}
\newtheorem{definition}{Definition}
\usepackage{psfrag}
\usepackage{eepic}
\usepackage{epic}
\usepackage{comment}
\usepackage{tikz}
\usetikzlibrary{arrows,shapes,calc}

\def\cC{\mathcal{C}}

\def\cJ{\mathcal{J}}
\def\cP{\mathcal{P}}
\def\bN{\mathbb{N}}
\def\bC{\mathbb{C}}
\def\bR{\mathbb{R}}
\def\bQ{\mathbb{Q}}
\def\bV{\mathbb{V}}
\def\bM{\mathbb{M}}
\def\bZ{\mathbb{Z}}
\def\Ex{\mathbb{E}}
\def\Var{\operatorname{Var}}
\def\orr{\overline{r}}
\def\cred{\color{red}}
\def\cblue{\color{blue}}
\def\DaGreen{\color{green}}\def\cblue{\color{blue}}
\def\cblack{\color{black}}
\def\eq1{\stackrel{1}{=}}
\definecolor{Violet}      {cmyk}{0.79,0.88,0,0}

\title{Bounded Discrete Bridges}
\author{Pierre Nicod\`eme\\[1ex]
\centering{
(Retired from CNRS - LIPN - University Paris Sorbonne
  Nord)
}}

\begin{document}

\maketitle
\begin{abstract}
In ~\cite{BaNi2010} Banderier and Nicod\`eme consider the height of
bounded discrete bridges and conclude to a limiting Rayleigh
distribution.
This result is correct although their proof is partly erroneous. They make asymptotic
simplifications based upon dominance properties of the roots of the
kernel of the walk within a disk centered at the origin, but these
dominance properties apply only upon a positive real segment.
However the very good agreement of simulations with their asymptotic expansion of the
probability distribution in case of {\L}ukasiewicz bridges
let us think that their proof could be corrected.
This is the scope of the present article which provides
 a proof using the dominance property only in its
domain of validity. We also consider the case of periodic walks,
a topic not considered in~\cite{BaNi2010}. 
\end{abstract}
 
\begin{remark}
To be self-content, we recall in this article the relevant 
 steps~\footnote{Sections~\ref{sec:intro},~\ref{sec:prelim},~\ref{sec:aperiodicg}
 come from~\cite{BanderierFlajolet2002}. 
Sections~\ref{sec:bridgoverh},~\ref{sec:asympsimp} (in part)
 follow~\cite{BaNi2010}.}   of proofs
of Banderier-Flajolet~\cite{BanderierFlajolet2002}
and Banderier-Nicodeme~\cite{BaNi2010}.
\end{remark}

\section{Introduction}
\label{sec:Introduct}
At the beginning of the years 2000 
 powerful methods of DNA
sequencing have been devised. (See~\cite{GGGF20}).
This allowed to search for genes responsible for a given
 disease by a method called ``Gene Set Enrichment Analysis'' (GSEA for
 short).
Considering one population of $g$ sick individuals and a background
 of $G$ unaffected individuals, when searching if a set of  genes
$\mathcal{G}$ is responsible for the disease, GSEA build a sorted list
$L_S$  of level of expression $e_g$ and $e_G$ of the genes of $g$ and
 of $G$ while remembering to which set the item comes from. Erasing
 next the expression levels from $L_S$ provides a word $w_{a_g,a_G}$
 such that
\[w_{a_g, a_G}\in
W_{a_g,a_G}=(a_g+a_G)^{\star}\Big|_{|a_g|=g,|a_G|=G}.\]
As instance,
 ${\displaystyle\left .\begin{array}{ll}\{e_G\}&=\{0.1, 0.5, 0.3, 0.7,
    0.2\}\\ \{e_g\}&=\{0.08, 0.16\}\end{array}\right|}$
 gives $ w_{a_g, a_G}= a_ga_Ga_ga_Ga_Ga_Ga_G$.
 \\
The more (resp. the less) the targeted  genes of individuals of population $g$ are expressed
with respect to those of population $G$, the more (resp. the less)
letters $a_g$ are over- \ (resp. under-) represented at the left part of words
$w_{a_g,a_G}$.

Using weights $v(a_g)$ and $v(a_G)$, a running sum statistics
$RS_k$ can be used to
evaluate this over-\ (under-)expression. Keller {\it et
  al.}\cite{KeBaLe07}  proposed in 2007 the following model:
\begin{align}
\label{def:gG} RS_k&=\sum_{1\leq j \leq k} X_j, \quad \text{where } \left\{
\begin{array}{ll} \Pr(X_j=+G) = v(a_g)=\dfrac{g}{G+g} \text{ if $a_g$ is read at position $j$}\\
\Pr(X_j=-g)= v(a_G)=\dfrac{G}{G+g} \text{ elsewhere}
\end{array}\right.\\
\label{def:Kel}
& RS_{g+G}=\sum_{1\leq j \leq g+G} X_j=0,\quad
\Var(X_j):=\sigma^2=Gg, \quad \Var(RS_k)=k\sigma^2.
\end{align}                
Therefore, the largest the maximum 
$M_{G,g}=\max_{1\leq k \leq g+G}RS_k$,
the more the genes of population $g$ are expressed with respect to
those of population $G$. 

Keller {\it et al.} provide a dynamical programming
approach to compute the probability that $M_{G,g}$ exceeds a given
$p$-value with memory requirement $g$ and running time $G(G+g)$.

The model of Keller {\it et al.} may be considered as a directed
random walk constrained to terminate at altitude $0$, and therefore as
a random
bridge, in which $v(a_g)$ and $v(a_G)$ are the probability of
moving on the half plane $(\bN\otimes \bZ)$ upwards and downwards
 while adding $+1$ 
to the abscissa of the current position at each unit time.

In the same period, years $2000$ and ahead, very important progresses
have been made in the study of random walks by analytic methods eventually
combined with powerful computer algebra tools. As instance, walks
constrained to remain in the upper right quarter of plane and with
steps set $\mathcal{S}\subset \{(\pm 1,\pm 1)\}$ have been characterized along the
mathematical properties of their generating functions, algebraic,
differentially finite, or transcendental.

The simplest walk is the Bernoulli walk with  steps set
$\{(1,Y_j)\}=\{(1,\pm 1)\}$ and
probability $1/2$ of choosing $Y_j=+1$ or $Y_j=-1$.
The sum $A_k=\sum_{1\leq j \leq k}Y_j$ can only be zero if $k$ is even. 
We consider therefore in this context bridges of
length $2n$.

For a given height $h$, the D\'esir\'e Andr\'e reflection which dates back to the $19$th
century selects the first value $k$ such that $A_k=h$ and replaces
each point $(j,A_j)$ with $j\geq k$ by the point $(j,2h- A_j)$,
taking the symmetry with respect to the half line $[(k,h),(+\infty,h)[$ of the
part of the walk at the right of the point $(k,A_k)$. There is a
one-to-one correspondence of bridges reaching height $h$ and walks
starting at zero and ordinate $2h$ at time $2n$.

\ \noindent
This gives  
${\displaystyle
\Pr\left(\max_{0 \leq k\leq 2n} A_k \geq
h\right)=\left.\binom{2n}{n+h}\right/\binom{2n}{n},}$
where ${\displaystyle \binom{2n}{n}}$ is the number of bridges of
length $2n$.

We have 
$n!=\Gamma(n+1)$, and $\Var A_k=k$ while $\Ex(A_k)=0$; therefore,  with $h=x\sqrt{n}$ , we get
\begin{equation}
\label{eq:limdesire}
\lim_{n\rightarrow\infty}\Pr\left(\max_{1\leq k \leq 2n} A_k>h\right)=\lim_{n\rightarrow
  \infty}\left.\frac{\Gamma(2n+1)}{\Gamma(n-x\sqrt{n}+1)\Gamma(n+x\sqrt{n}+1)}\right/\binom{2n}{n}=e^{-2x^2}.
\end{equation}
On the other hand, if $B_t$ is the standard Brownian motion,
$\Var B_t=t^2$ implies that $\Var(B_{k/2n})=k^2/(2n)^2$, and since
 $\Var RS_k=k\sigma^2$, without claim of rigor, we have for $k<n$ large enough
\[
\lim_{n\rightarrow\infty}\frac{1}{\sigma\sqrt{2n}}
(RS_k,RS_{k+1},\dots,RS_{2n-1},RS_{2n})
=
(B_{k/2n},B_{(k+1)/2n},\dots,B_{(2n-1)/2n},B_{2n/2n})\]
\begin{figure}[hb!]
\setlength{\unitlength}{1.1mm}
\centering\begin{picture}(60,15)
\put(-30,-38) {\includegraphics[height=7.5cm,width=7.5cm]{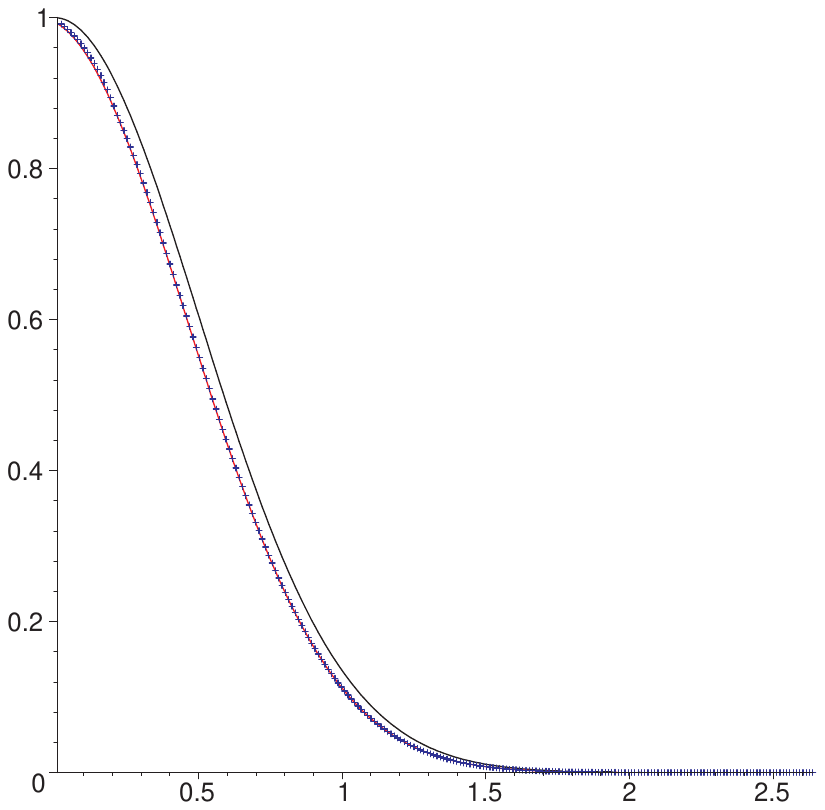}}
\put(30,-38) {\includegraphics[height=7.5cm,width=7.5cm]{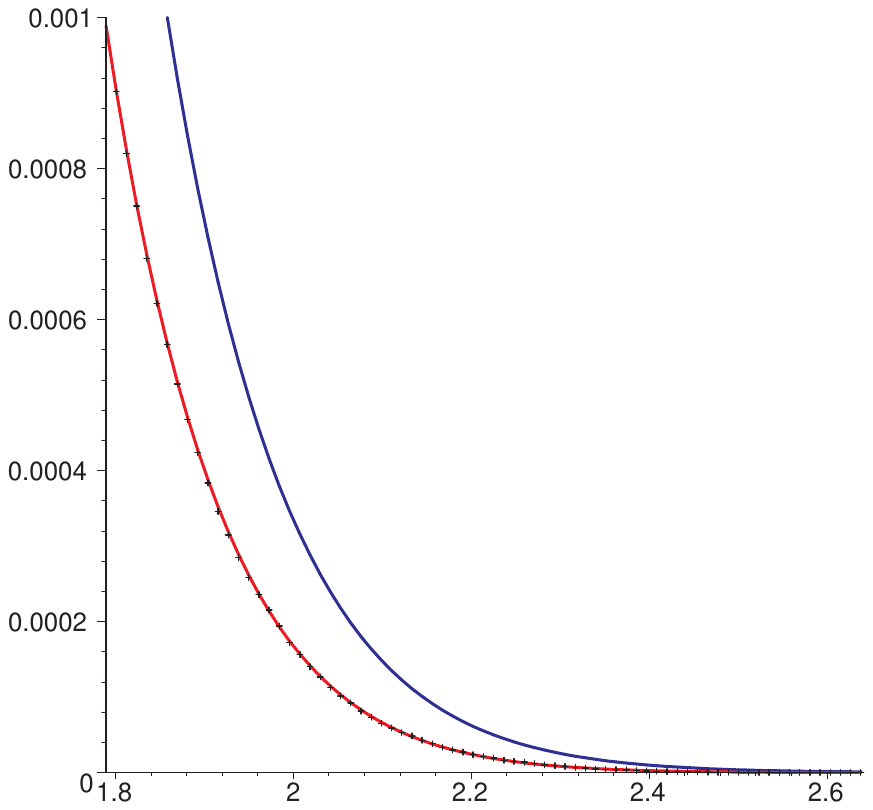}}
\put(-30,-78) {\includegraphics[height=7.5cm,width=7.5cm]{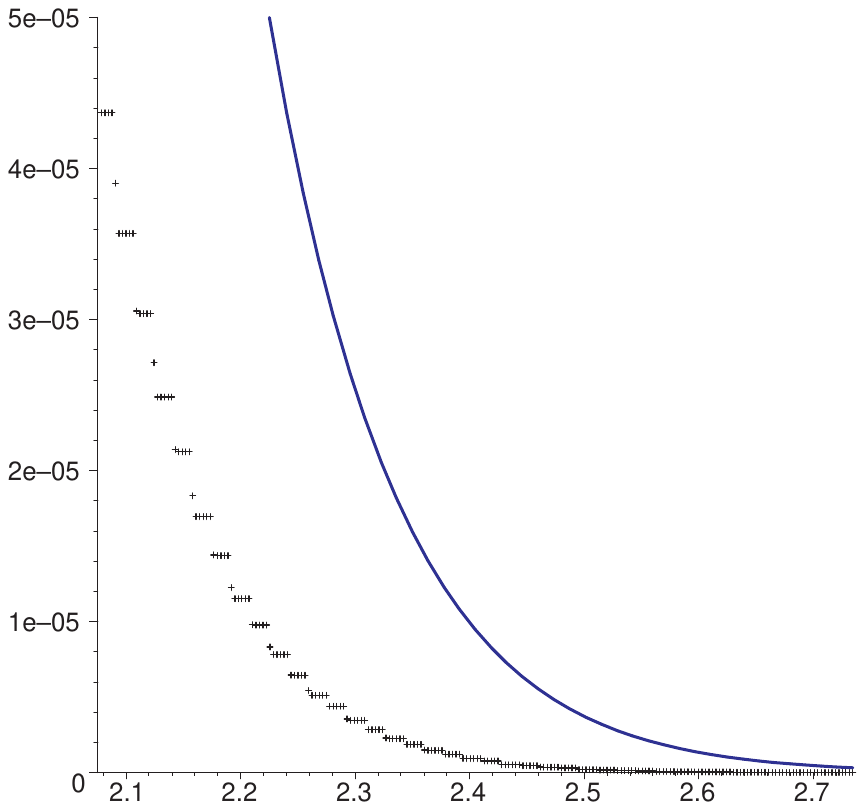}}
\put(30,-78) {\includegraphics[height=7.5cm,width=7.5cm]{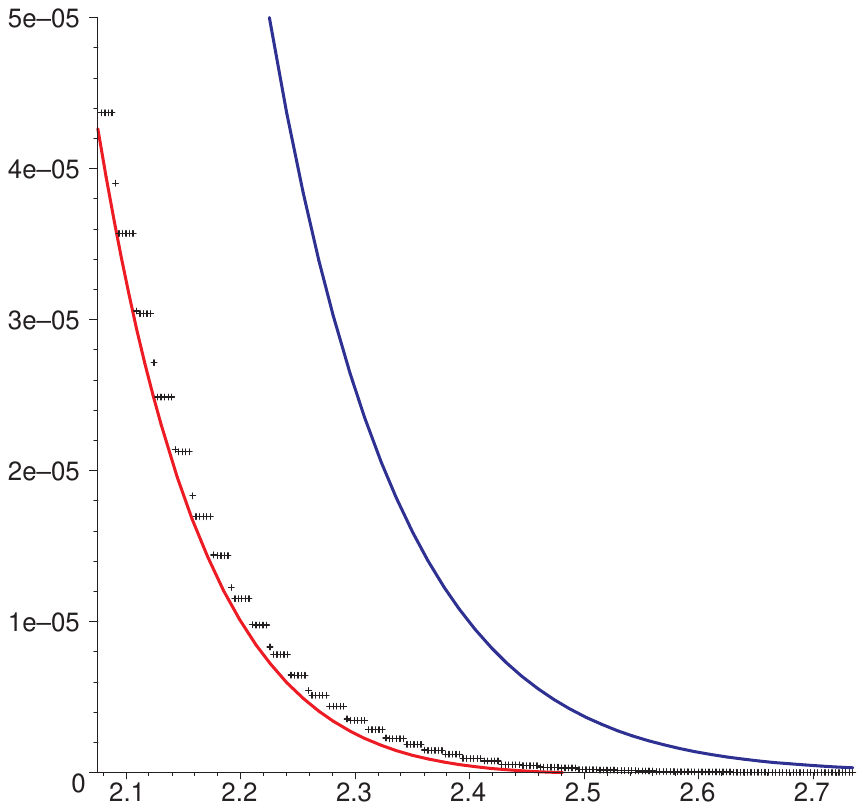}}
\put(16,-4){\scriptsize $n=400$}
\put(16,-14){\scriptsize $\frac{h}{\sqrt{nGg}}$}
\put(-25,16){\scriptsize $\Pr(\max_{1\leq j\leq n}RS_j \geq h)$}
\put(0,5){\scriptsize $G=19, \ g=1$}
\put(78,-4){\scriptsize $n=400$}
\put(78,-14){\scriptsize $\frac{h}{\sqrt{nGg}}$}
\put(45,16){\scriptsize $\Pr(\max_{1\leq j\leq n}RS_j \geq h)$}
\put(70,5){\scriptsize $G=19, \ g=1$}
\put(16,-44){\scriptsize $n=400$}
\put(16,-54){\scriptsize $\frac{h}{\sqrt{nGg}}$}
\put(-25,-26){\scriptsize $\Pr(\max_{1\leq j\leq n}RS_j \geq h)$}
\put(0,-40){\scriptsize $G=11/93, \ g=1$}
\put(78,-44){\scriptsize $n=400$}
\put(78,-54){\scriptsize $\frac{h}{\sqrt{nGg}}$}
\put(45,-26){\scriptsize $\Pr(\max_{1\leq j\leq n}RS_j \geq h)$}
\put(70,-40){\scriptsize $G=11/93, \ g=1$}
\end{picture}
\ \vspace*{65mm}
\caption{\label{fig:BaNisimu}Plots of the Rayleigh distribution (upper
continuous curves),
 simulations (dots) and asymptotics expansions obtained by
 Equation~(27) of~\cite{BaNi2010} (lower continuous red curve) for the
distribution of the height $h$ of the walk. (Plots for $G=19$ come 
from~\cite{BaNi2010}).
The plots on top right and below are zooms for
small $p$-values. (Top) These plots show a perfect concordance of the
simulations
with the discrete asymptotics obtained by~\cite{BaNi2010}, and a clear discrepancy
with the Rayleigh distribution. (Bottom) Heuristics computations with fractional $G$ to simulate $G=93$
and $g=11$}
\end{figure}
\ \\\\
The reflection principle applies also to
 the standard Brownian bridge $B_t^0$ which gives a Rayleigh law
for  its height ${\displaystyle \Pr(\max_{0\leq t \leq 1}B_t^0>x)=e^{-2x^2}}$.
See Durrett~\cite{Durrett91}
(Section 7.8).

We therefore expect a Rayleigh law as limit law for the height of
discrete bridges.
However, the one-to-one correspondence of the reflection principle is
no more verified for arbitrary jumps. Moreover, for practical
applications,  we remain far away from the Rayleigh law (see Figure~\ref{fig:BaNisimu}).

Equation (27) of \cite{BaNi2010} is equivalent up to a low order term 
to Equation (\ref{eq:getorder3}) of the present article. It evaluates
\begin{equation}
 \Pr\left(\max_{k\leq 2n}\sum_{1\leq i\leq k}X_i >x\sigma\sqrt{2n}=e^{-2x^2}
   \quad\Big|\quad X_{2n}=0\right)
\end{equation}
The plots of Figure~\ref{fig:BaNisimu} show a very good agreement of
the simulations with the asymptotic results
of~\cite{BaNi2010}. But, during the course of their proof, Banderier-Nicodeme
consider a double sum ${\displaystyle 
B(z)=\sum_{1\leq k \leq d}\sum_{1\leq j\leq
  c}\left(\frac{u_j}{v_k}\right)^h\frac{Q_k(u_j)}{
Q_k(v_k)}\frac{u'_j}{v_k}}$ where $u_j:=u_j(z)$ (resp. $v_k:=v_k(z)$)
are the small (resp. large) roots of the characteristic equation of
the walk; they assume next that along any Cauchy circuit 
$\mathcal{C}\subset \{|z|<\rho\}$, where $\rho$    is the radius of
convergence of the considered generating function, there exists a domination property,
\begin{equation}
\label{eq:falsedom}
|u_j(z)|<u_1(z)<v_1(z)<|v_k(z)|,
\end{equation}
where $u_1(z)$ (resp. $v_1(z))$ is the small (resp. large) real root of the
characteristic equation, and they consider that 
$|u_j/u_1|^h\rightarrow 0$ and $|v_1/v_k|^h\rightarrow 0$
as $h\rightarrow\infty$. This leads them to an expression 
${\displaystyle 
B(z)=\left(\frac{u_1}{v_1}\right)^h\frac{Q(u_1)}{Q_1(v_1)}
\frac{u'_1}{v_1}}$.
However Equation (\ref{eq:falsedom}) is not verified
for any circuit $\mathcal{C}$. It was challenging to find a circuit
upon which the domination property is verified and this is the scope
of the present article. Moreover we also consider the periodic walks
that
 Banderier-Nicodeme~\cite{BaNi2010}
 have left aside. 

Banderier-Nicodeme provide detailed asymptotics only for the {\L}ukasiewicz
bridges where there is only one negative jump $-1$;

Although the heuristics of fractional jumps is satisfactory for the
case $G=93$ and $g=11$ (Figure~\ref{fig:BaNisimu} below and right), we
do not provide an
 asymptotic solution for
the height 
of walks with
two jumps, one positive and one negative, corresponding to the approach of Keller
{\it et al.}. See Equations (\ref{def:gG},\ref{def:Kel}).

\section{Definitions}
\label{sec:intro}
This article is organized as follows.
\begin{enumerate}
\item We first compute $F^{[>h]}(z,u)$, the bivariate generating
  function
giving the probability that a walk of length $n$ exceeds height $h$.

Next, we compute $B(z)=[u^0]F^{[>h]}(z,u)$, the restriction of these
walks to bridges.

\item We extract the Taylor coefficient of order $n$ of
  $B(z)$.
We  cope first  with the aperiodic case and next with the
periodic case.
In both cases, the proof has two steps.
\begin{enumerate}
\item Design of a Cauchy contour upon which the domination properties
  of the roots of the kernel of the walk applies, which allows
  asymptotic
simplifications.
\item Application of the singularities analysis methods as exposed in
  Flajolet-Sedgewick book~\cite{FlajoletSedgewick2009}; in particular
   use of the semi-large powers approach and of Hankel integrals.
\end{enumerate}
\item A section is dedicated to {\L}ukasiewicz bridges for which
  asymptotic expansions at higher order is available; we mention there
  the occurrence of Hermite polynomials in the expansions.
We use in this section Newton iterations and do a numerical check of
our expansions  for Dyck walks.
\end{enumerate}

\section{Preliminaries and definitions}
\label{sec:prelim}
We recall the definitions of
Banderier-Flajolet~\cite{BanderierFlajolet2002}.
\begin{definition}
\label{def:jumps}
We consider simple directed walks defined by  sets of jumps 
$\mathcal{S}\in\{d,d-1,\dots,-c+1,-c\}$ and  sets of weights,
$\mathcal{W}\in\{p_d,p_{d-1},\dots,p_{-c+1},p_{-c}\}$ with $d\neq 0$ and
$c\neq 0$.

The characteristic Laurent polynomial $P(u)$ of a walk with set of
jumps
$\mathcal{S}$ and weights $\mathcal{W}$ verifies
\begin{equation}
\label{eq:charpoly}
P(u)=p_d u^d
+p_{d-1}u^{d-1}+\dots+p_1u+p_0+\frac{p_{-1}}{u}+\dots+\frac{p_{-c}}{u^c},
\end{equation}
where the coefficients $p_i$ are positive rational numbers.
\begin{equation}
\label{eq:kerneleq}
\text{The equation}\quad 1-zP(u)=0,\quad \text{ or equivalently }\quad u^c-zu^cP(u)=0,
\end{equation}
is the \emph{kernel equation}, the quantity $K(z,u)=u^c-zu^cP(u)$
being referred to as the \emph{kernel} of the walk.
\end{definition}

\begin{definition}
\label{def:period}
A Laurent series $h(z)=\sum_{n \geq -a} h_nz^n$ is said to admit period
  $p$ if there exists a Laurent series $H$ and an integer $b$ such
that
\begin{equation}
\label{eq:perseries}
h(z)=z^bH(z^p);
\end{equation}
the largest $p$ such that a decomposition (\ref{eq:perseries}) holds
is called the period of $h$. The series is \emph{aperiodic} if the period is
$1$.

A simple walk defined by the set of jumps $\mathcal{S}$ is said to have
\emph{period} $p$ if the characteristic polynomial has period $p$.

A simple walk is said to be \emph{reduced} if the gcd of jumps is
equal to $1$.
\end{definition}
For a bounded walk at height $h$, and $(x_n,y_n)$  its position at
time $n$ within the lattice $\bN \otimes \bZ$, the possible positions $(x_{n+1},y_{n+1})$ at
time $n+1$ are
\begin{align*}
   &\quad x_{n+1}=x_{n}+1,\\
   &\quad y_{n+1}=y_n+j \quad\text{if}\quad y_n+j\leq h,\quad j\in
  \mathcal{S},\\[2ex]
&\text{with} \quad(x_0,y_0)=(0,0),\qquad \text{ (the walk starts at the origin).}
\end{align*}
\begin{remark}
\label{rem:reduced}
If a non-reduced walk verifies $\operatorname{gcd}\mathcal{S} = r$,
the points accessible by the walk lie on the sub-lattice $\bN\otimes
r\bZ$, and by a linear change of abscissa, the walk can be reduced.
We assume therefore in the following that the walks we consider are reduced.
\end{remark}
\section{Aperiodic case}
\label{sec:aperiodicg}
\subsection{Unbounded bridges}
\label{sec:unbaperiodic}
Banderier-Flajolet~\cite{BanderierFlajolet2002} 
compute asymptotically the number of bridges of length $n$.
They use a saddle-point
integral
at the singular point $\tau$ such that $P'(\tau)=0$ and justify it by
the aperiodicity which implies that $|P(u)|$ is only maximal at
$z=\tau$. This leads to the following
theorem.
\begin{theorem}[Banderier-Flajolet 2002~\cite{BanderierFlajolet2002}-Theorem 3]
\label{theo:BaFlaaperunb}
Let $\tau$ be the structural constant of an aperiodic walk determined
by $P'(\tau)=0$. The number $V_n$  of bridges of size $n$ admits a complete
asymptotic expansion
\begin{equation}
V_n\sim \lambda_0\frac{P(\tau)^n}{\sqrt{2\pi n}}
     \left(1+\frac{a_1}{n}+\frac{a_2}{n^2}+\dots\right),
     \qquad \lambda_0=\frac{1}{\tau}\sqrt{\frac{P(\tau)}{P''(\tau)}}.
\end{equation}
\end{theorem}
 We follow the proof of
 Banderier-Flajolet. (See also Greene and Knuth~\cite{GreKnu81}).

Let $V_n$ be the number of bridges of aperiodic walks of length $n$.
The large power $P(u)^n$ has a saddle point at $\tau$ such that $P'(\tau)=0$, and therefore
\begin{align}
\label{eq:unboun0}
&V_n=[u^0]P(u)^n=\frac{1}{2\pi i}\oint_{|u|=\tau} \frac{P(u)^n}{u}du
\\[2ex]
\nonumber
&\sim \frac{1}{2\pi}\int_{\tau e^{-i\epsilon}}^{\tau e^{+i\epsilon}}
\exp\left(n\left(\log P(\tau)+\frac{1}{2}\frac{P''(\tau)}{P(\tau)}(u-\tau)^2+\dots
\right)\right)\frac{du}{u}\\[2ex]
\nonumber
&\sim \frac{P(\tau)^n}{2\pi}\int_{
  -\infty}^{+\infty} e^{-n\phi
  t^2/2}G(t,n)\frac{dt}{\sigma\sqrt{n}},\qquad\qquad\qquad 
\left\{\begin{array}{l}u=\tau e^{i
  t},\\\phi=P''(\tau)/P(\tau),\ t=s/\sqrt{\phi n},\end{array}\right.\\[2ex]
\nonumber
&\sim
\frac{P(\tau)}{\tau2\pi\sqrt{\phi}}\int_{-\infty}^{+\infty}e^{-s^2/2}(1+H(s,n))\frac{ds}{\sqrt{n}}
=\frac{P(\tau)}{\tau\sqrt{2\pi n\phi}}
\times\left(1+O\left(\frac{1}{\sqrt{n}}\right)\right),\quad
\nonumber
\end{align}
where 
\[H(s,n)=\exp\left(n\sum_{i\geq
  3}\alpha_i\times\frac{s^i}{
  n^{i/2}}\right)=\sum_{j\geq
  3}\beta_{j}\frac{s^{j}}{n^{(j-2)/2}}.\]
We recognize the Gaussian integrals,
\[ \frac{1}{\sqrt{2\pi}}\int_{0}^{\infty}e^{-s^2/2}s^{2k+1}ds=0,\qquad 
 \frac{1}{\sqrt{2\pi}}\int_{0}^{\infty}e^{-s^2/2}s^{2k}ds=\frac{k!}{(k/2)!2^{k/2}}.\]
At order 4, in the particular case where $P(1)=1$, with  $\sigma=\sqrt{P''(1)},\ \xi=P'''(1)$ and
$\theta=P''''(1)$, this gives
\begin{equation}
\label{eq:ucond}
b_n^{<\infty}= \frac{1}{\sigma\sqrt{2\pi n}}
        \left(1-\frac{1}{n^2}\frac{1}{\sigma^6}\left(\sigma^4+\frac{1}{2}\sigma^2\xi-\frac{1}{8}\sigma^2\theta+\frac{3}{8}\sigma^6+\frac{5}{24}\xi^2\right)+O\left(\frac{1}{n^4}\right)\right),
\end{equation}
where $b_n^{<\infty}$ is the probability that a walk of length $n$ is  a
bridge.

\subsection{Bridges reaching height  $h$}
\label{sec:bridgoverh}
Similarly to Banderier-Flajolet~\cite{BanderierFlajolet2002}, Equations
(14) and (16) and Banderier-Nicodeme~\cite{BaNi2010}, 
if  $f_n(u)=\sum_{j}f_{n,j}u^j$ where $f_{n,j}$ counts the number of  walks at
time $n$ with ordinates $y\leq j$, we get
\begin{equation}
\label{eq:fnu}
f_{n+1}(u)=f_n(u)P(u)-\sum_{k=1}^du^{h+k}[u^{h+k}]f_n(u)p_ku^k,\qquad f_0(u)=1.
\end{equation}
Summing up over $n$ provides the generating function of the
corresponding walks
\begin{align}
\nonumber
F^{[\leq h]}(z,u)&=1+\sum_{n\geq 0}f_{n+1}(u)z^{n+1}=1+zP(u)F^{[\leq
    h]}(z,u)-z\{u^{>h}\}P(u)F^{[\leq h]}(z,u),\\
\label{eq:Fzu}
&=1+zP(u)F^{[\leq h]}(z,u)-z\sum_{k=1}^{d}u^{h+k}F_k(z)=\sum_{n\geq
  0}\sum_{-nc\leq j \leq h}f_{n,j}u^j z^n,
\end{align}
where, at time $n$,
\begin{itemize} 
\item if the set of weights $\mathcal{W}$ is a probability
  distribution, {\it i.e}
$\sum p_i=1$, the quantity 
 $f_{n,j}$ is the probability of reaching height $j$ 
  at time $n$
\item or, elsewhere, $f_{n,j}$  is
 the number of ways of reaching level $j$ if the non-zero coefficients
 $p_i$ have value $1$.
\end{itemize}
We obtain the equation
\begin{equation}
\label{eq:syslin}
(1-zP(u))F^{[\leq h]}(z,u)=1-z\sum_{k=0}^{d-1}u^{h+k}F_k(z),
\end{equation}
where the functions $F_k(z)$ are unknown.

We use the most basic kernel method, and the kernel of the walk $K(z,u)=1-zP(u)$ is cancelled by 
$d$ large roots $v_1(z),\dots,v_d(z)$, and $c$ small roots $u_1(z),\dots,
u_{-c}(z)$ that verify
\begin{equation}
\label{eq:limroots}
z\rightarrow 0 \Longrightarrow \left\{\begin{array}{lll}
       v_k(z)\rightarrow p_d^{-1/d}\varpi_k z^{-1/d},&\quad
       \varpi_k=\exp(2\pi i(1-k)/d),&\quad k=(1,\dots,d)\\
       u_j(z)\rightarrow p_c^{1/c}\omega_j z^{1/c}&\quad
       \omega_j\ = \exp(2\pi i (j-1)/c),&\quad j=(1,\dots,c)\\[2ex]
       u'_j(z)\rightarrow \dfrac{p_c^{1/c}}{c}\omega_j z^{-1+1/c}.&
       \end{array}\right.
\end{equation}
We have $d$ unknowns $F_k(z)$ in  Equation (\ref{eq:syslin}), but the
$d$ large roots $v_i(z)$ with $i\in \{1,\dots,d\}$ provide a set of
$d$ linear equations
\[\left\{\begin{array}{l}
{ v_1(z)}^{h+1}{ F_{h+1}(z)}+\dots+{ v_1(z)}^{h+d}{ F_{h+d}(z)}=1/z,\\
\dots \\
{ v_d(z)}^{h+1}{ F_{h+1}(z)}+\dots+{ v_d(z)}^{h+d}{ F_{h+d}(z)}=1/z
\end{array}\right.
\]
Solving the system with the Cramer formula provides expressions
involving a determinant
 $\bM$ and  Vandermonde-like determinants $\bV$ and 
 $\bV_k$ of dimension $d$,
\begin{equation}
\bM=
\left|\begin{array}{ccccc}
   v_1^{h+d} &\dots & v_1^{h+k} & \dots & v_1^{h+1}\\
   \dots       &\dots           & \dots & \dots          &\\
   v_d^{h+d} &\dots & v_d^{h+k} & \dots & v_d^{h+1}
   \end{array}\right|=v_1^h\dots v_d^h\, \bV,\quad \text{with}\quad
\bV=\prod_{r=1}^{d-1}\prod_{s=r+1}^d
(v_r(z)-v_s(z)).
\end{equation}
This gives\footnote{We differ from Banderier-Nicodeme 2010~\cite{BaNi2010} who
consider only $Q_1(u)$.} with $\bV_k(x) =\left.\bV\right|_{v_k(z)=x}$
and ${\displaystyle Q_k(u)=\prod_{\stackrel{1\leq m\leq d}{m\neq
   k}}(u-v_m(z))=\sum_{m=0}^{d-1}q_{km}(z)u^m,}$
\begin{equation}
\label{eq:Piofi}
z F_k(z)=\frac{u^{h+1}}{v_k^{h+1}}\frac{\bV_k(u)}{\bV}
=\frac{u^{h+1}}{v_k^{h+1}} \frac{Q_k(u)}{Q_k(v_k)}
 \end{equation}
\normalsize
Since
\[ F^{<+\infty}(z,u)=\frac{1}{1-zP(u)}\quad\text{and}
    \quad F^{[>h]}=F^{<+\infty}-F^{[\leq h]},\]
where $F^{<+\infty}$ is the generating functions of unbounded walks, we get for $F^{[>h]}(z,u)$ the
generating function of walks going upon height $h$, with $v_j:=v_j(z)$,
\begin{equation}
\label{eq:flargerh}
F^{[>h]}(z,u)=\frac{1}{1-zP(u)}{\displaystyle\sum_{k=1}^d
  \left(\frac{u}{v_k}\right)^{h+1}\left(\frac{Q_k(u)}{Q_k(v_k)}\right)}\\
\end{equation}
Banderier-Flajolet 2002~\cite{BanderierFlajolet2002} provides an explicit
expression for paths terminating at height $m$. We use it for bridges,
or walks terminating at height $0$, which allows us to get rid of the
variable $u$.
\begin{theorem}[Banderier-Flajolet (2002)]
\label{theo:heightm}
The generating function $W_m(z)$ of paths terminating at altitude $m$ is, for
$-\infty<m<c$,
\[W_m(z)=[u^m]\frac{1}{1-zP(u)}=z\sum_{j=1}^c\frac{u'_j(z)}{u_j(z)^{m+1}}.\]
\end{theorem}
\noindent
Therefore, the generating function $B_h(z)$ of bridges reaching height
$h$ verifies
\begin{align}
B_h(z)&=[u^0]F^{[>h]}(z,u)
   =[u^0] \frac{1}{1-zP(u)}\sum_{k=1}^d 
   \frac{1}{Q_k(v_k)}\left(\frac{u}{v_k}\right)^{h+1}\sum_{m=0}^{d-1}q_{km}(z)u^m
\label{eq:BhF0}\\
&=[u^0] \frac{1}{1-zP(u)}\sum_{k=1}^d\sum_{m=0}^{d-1}
           \frac{q_{km}(z)}{Q_k(v_k)v_k^{h+1}}u^{h+1+m}\\
&=\sum_{k=1}^d
   \frac{1}{v_k^{h+1}Q_k(v_k)}\sum_{m=0}^{d-1}q_{km}(z)W_{-h-1-m}(z) 
\nonumber\\ 
&=\sum_{k=1}^d
   \frac{1}{v_k^{h+1}Q_k(v_k)}\sum_{m=0}^{d-1}q_{km}(z)z\sum_{j=1}^c
   u_j^{h}u_j^m u'_j
\label{eq:sumjk}
=z\overline{B}_h(z),
\end{align}
where
\begin{equation}
\label{eq:overB}
\overline{B}_h(z)=\sum_{k=1}^d\sum_{j=1}^c \left(\frac{u_j}{v_k}\right)^h \frac{Q_k(u_j)}{Q_k(v_k)}\frac{u'_j}{v_k}
\end{equation}
The computation of the Taylor coefficient $[z^n]B_h(z)$ by  a
Cauchy integral implies to localize the singularities of $B_h(z)$.

The singularities $\zeta$ of the roots of the kernel $K(z,u)=1-zP(u)$ drive the
asymptotic expansion
of the function $B_h(z)$ of Equation (\ref{eq:sumjk}). They verify
\begin{equation}
\label{eq:singul}
P'(\upsilon)=0,\qquad 1-\zeta P(\upsilon)=0
\end{equation}
We recall in the following remark properties of algebraic functions 
(see Stanley~\cite{Stanley1980}) that we apply to the roots $u_i(z)$ 
and $v_j(z)$ of the kernel equation.
\begin{figure}[!ht]
\setlength{\unitlength}{1.1mm}
\begin{picture}(60,0)(-40,20)
\put(-30,-38) {\includegraphics[height=7.5cm,width=7.5cm]{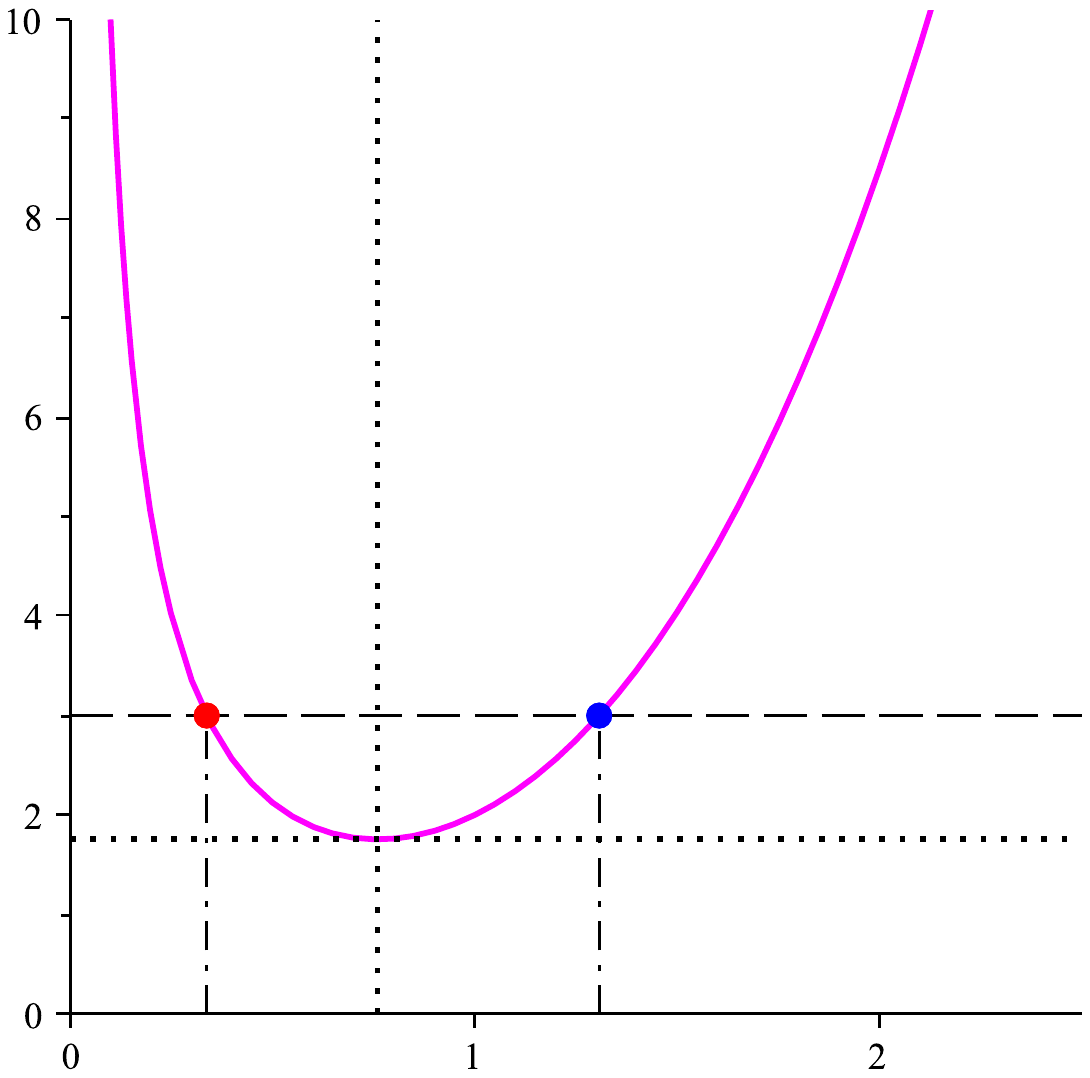}}
\put(30,-38) {\includegraphics[height=7.5cm,width=7.5cm]{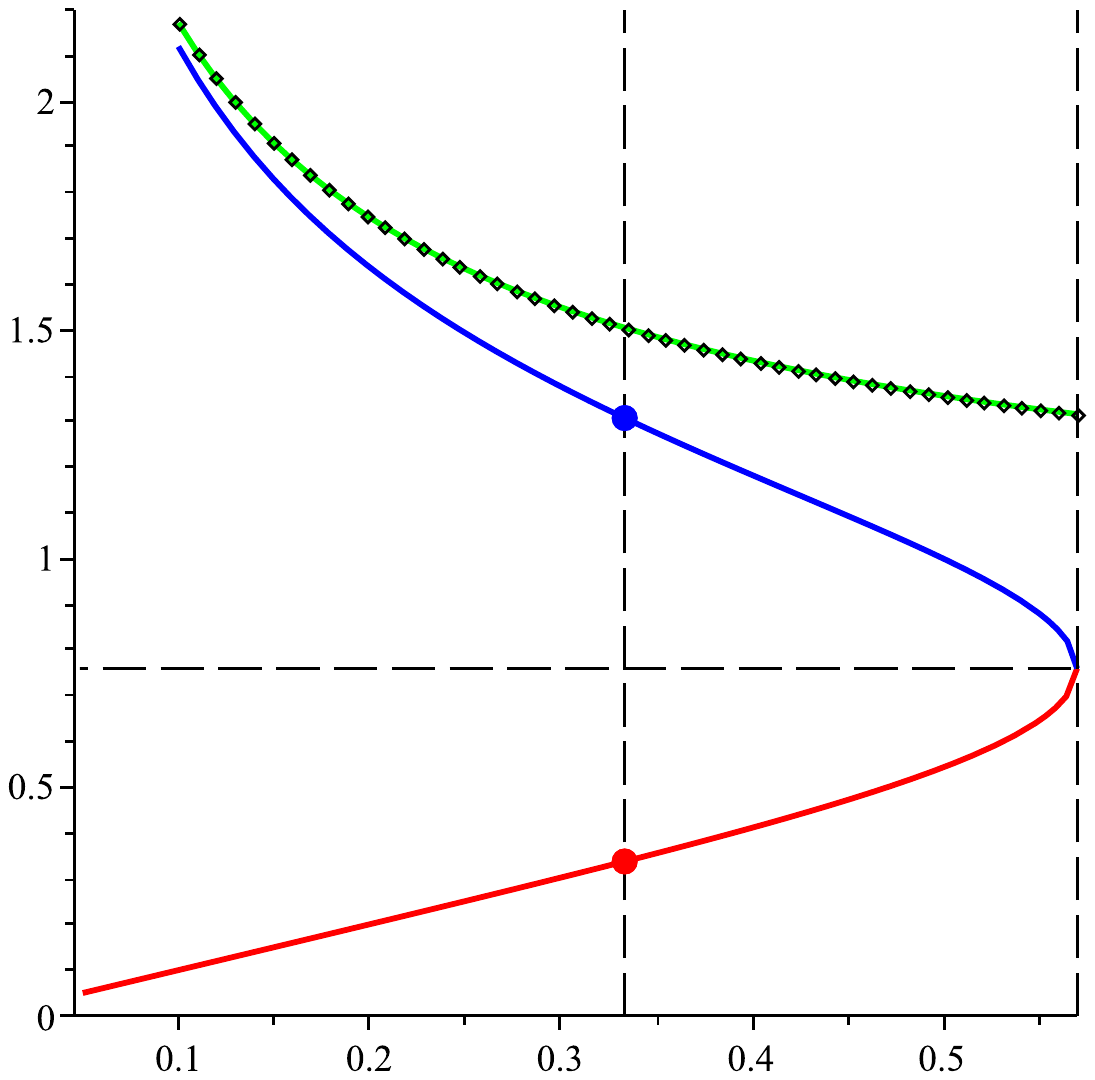}}
\put(91,-24.5){\tiny $z$}
\put(89.5,20){\tiny $z=\rho$}
\put(62,20){\tiny $z=\zeta=1/3$}
\put(36.5,-8.5){\tiny $\tau$}
\put(70,-18){\cred\tiny $u_1(z)$}
\put(70,-3){\cblue\tiny $v_1(z)$}
\put(73,6){\tiny ${\DaGreen|v_2(z)|},{\cblack|v_3(z)|}$}
\put(30,-21){\tiny $u$}
\put(-26,-11){\tiny $1/\zeta$}
\put(-16.5,-24.8){\tiny $u_1(\zeta)$}
\put(4.5,-24.8){\tiny $v_1(\zeta)$}
\put(-5.5,-24.8){\tiny $\tau$}
\put(12,20){\tiny $P(u)=u^3+\dfrac{1}{u}$}
\put(20,-14.5){\tiny $P(\tau)=1/\rho$}
\put(-40,-82){\rule{\textwidth}{0.2mm}}
\end{picture}
\ \vspace{5cm}
\caption{\label{fig:ROOTS} A  visual rendering of the proof of the domination
  property~\cite{BanderierFlajolet2002} stated in
  Lemma~\ref{lem:domin} of Banderier-Flajolet  for $P(u)=u^3+\dfrac{1}{u}$.
(Left): behaviour of the characteristic polynomial
$P(u)$. (Right): domination
property of the roots in the real interval $]0,\rho]$. We have
$P''(u)>0$ for $u>0$, while $P(u)$ tends to infinity as $u$ tends to
$0$ or $+\infty$. There exists
a number $\tau$ that
is the unique positive solution of $P'(z)=0$. For $\frac{1}{z}>\frac{1}{\rho}$
or $z<\rho$ with $\rho=\frac{1}{P(\tau)}$ the equation $1-zP(u)=0$ has 
for $z\in ]0,\rho[$ two
{\sl real} solutions $u_1(z)$ and $v_1(z)$ such that {\it(i)} 
$\lim_{z\rightarrow 0^+}u_1(z)= 0$ (dominant small root) and $\lim_{z\rightarrow 0^+}v_1(z)= +\infty$
(dominant large root) and 
{\it(ii)} $u_1(z)< v_1(z)$ for $u\in[0,\rho[$.
As proved in Lemma~\ref{lem:domin} we have  $u_1(z)<v_1(z)<
|v_2(z)|=|v_3(z)|$
 for $z\in ]0,\rho[$; moreover for the present example $v_2(z)$ and
    $v_3(z)$
 are algebraically    conjugate.
}
\end{figure}
\begin{remark}
\label{rem:algebrprop}
({\it i}) The derivative of an algebraic function is an algebraic function,
  and so are the $u'_j(z)$.
({\it ii}) A rational expression of algebraic functions such as $B_h(z)$ is
  algebraic; in particular  the denominators $Q_k(z)$ vanish at the
  intersections $\upsilon_{rs}$ of two roots $v_r(z)$ and $v_s(z)$ of the
  kernel,  points which verify $P'(\upsilon_{rs})=0$ and $1-zP(\upsilon_{rs})=0$
  and are algebraic points.
\end{remark}
This implies that the singularities of $B_h(z)$ are the singularities
of the roots of the kernel.

We consider them in two steps, {\it(i)} by use of Lemma 2 of domination of the
roots 
(Banderier-Flajolet~\cite{BanderierFlajolet2002}), which will 
further allow us to apply to
$B_h(z)$ asymptotic simplifications, {\it(ii)} by use of a domination property of
$B_h(z)$ by the generating function of unbounded walks $1/(1-zP(u))$.

\begin{lemma}[Banderier-Flajolet Lemma 2 (2002)]
\label{lem:domin}
Let $\tau$ verify $P'(\tau)=0$ and $\rho=1/P(\tau)$. For an aperiodic walk, the principal small branch $u_1(z)$ is analytic
on the open interval $z\in (0,\rho)$. It dominates strictly in modulus
all the other small branches $u_1(z),
\dots, u_c(z)$, throughout the half-closed interval $z\in (0,\rho[$.

By duality, the large roots $\widetilde{v}_j(z)$ of the kernel for $\widetilde{P}(u)=P(1/u)$ are the small
roots
of the kernel for  $P(u)$.
Therefore, for $z\in]0,\rho[$ the small (resp. large) roots $u_i(z)$
  (resp. $v_j(z)$) verify
\begin{equation}
\label{eq:domin}
|u_i(z)|<u_1(z)<v_1(z)<|v_j(z)|,\qquad (i\neq 1,\ j\neq 1).
\end{equation}
\end{lemma}
\begin{proof}[Sketch of proof.]
We have~\footnote{See Figure~\ref{fig:ROOTS}}  by the triangle inequality
\begin{equation}
\label{eq:triangle} |P(re^{it})|<P(r)\quad\text{for } 0<r<\rho\text{ and }
t\not\equiv 0 \  {(\!\!\!\!\!\mod 2\pi)}.
\end{equation}
For $z=x$ real and $0<x<\rho$
and $w$ any root of
$1-xP(w)=0$ that is at most $\tau$ in modulus and not equal
to $u_1(x)$ 
(not real and positive), we have by
(\ref{eq:triangle})
\[x=\frac{1}{P(u_1(x))}=\frac{1}{P(w)}>\frac{1}{P(|w|)},\]
which implies $|w|<u_1(x)$ since $1/P$ is increasing in $[0,\tau]$.
 \end{proof}
\begin{figure}[!t]
\setlength{\unitlength}{1.1mm}
\begin{picture}(60,0)(-40,20)
\put(-45,-38) {\includegraphics[height=5.5cm,width=5.5cm]{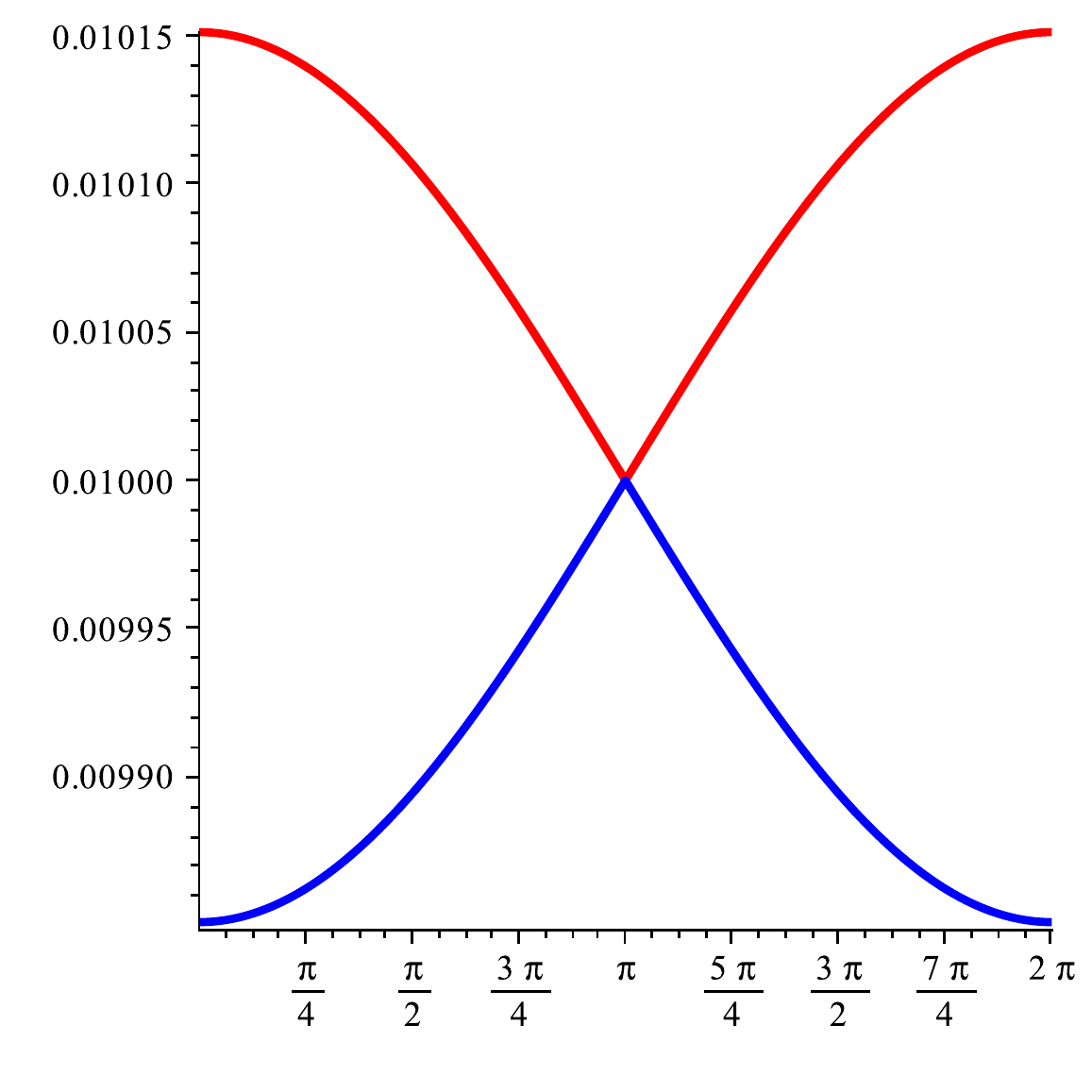}}
\put(5,-38) {\includegraphics[height=5.5cm,width=5.5cm]{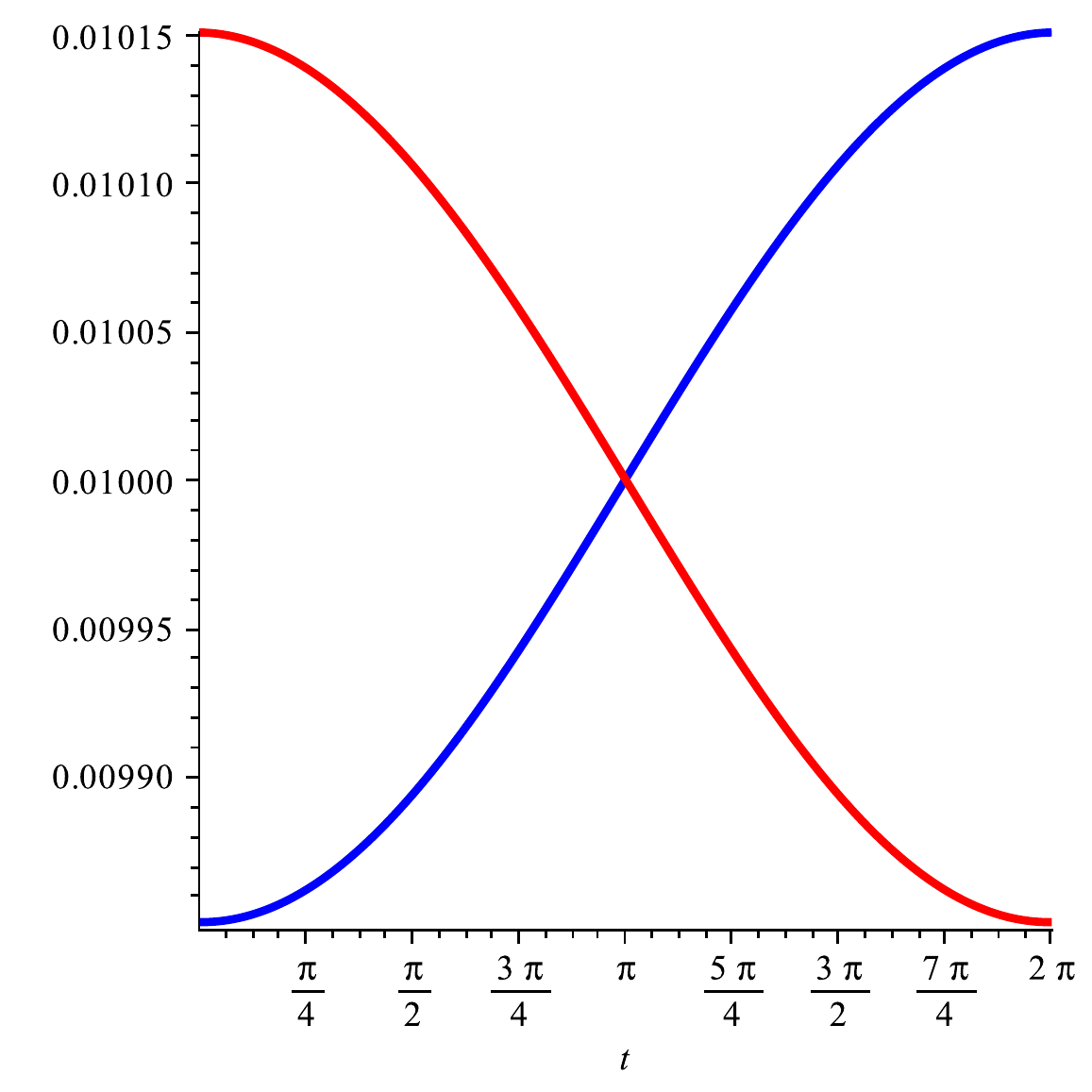}}
\put(55,-38) {\includegraphics[height=5.5cm,width=5.5cm]{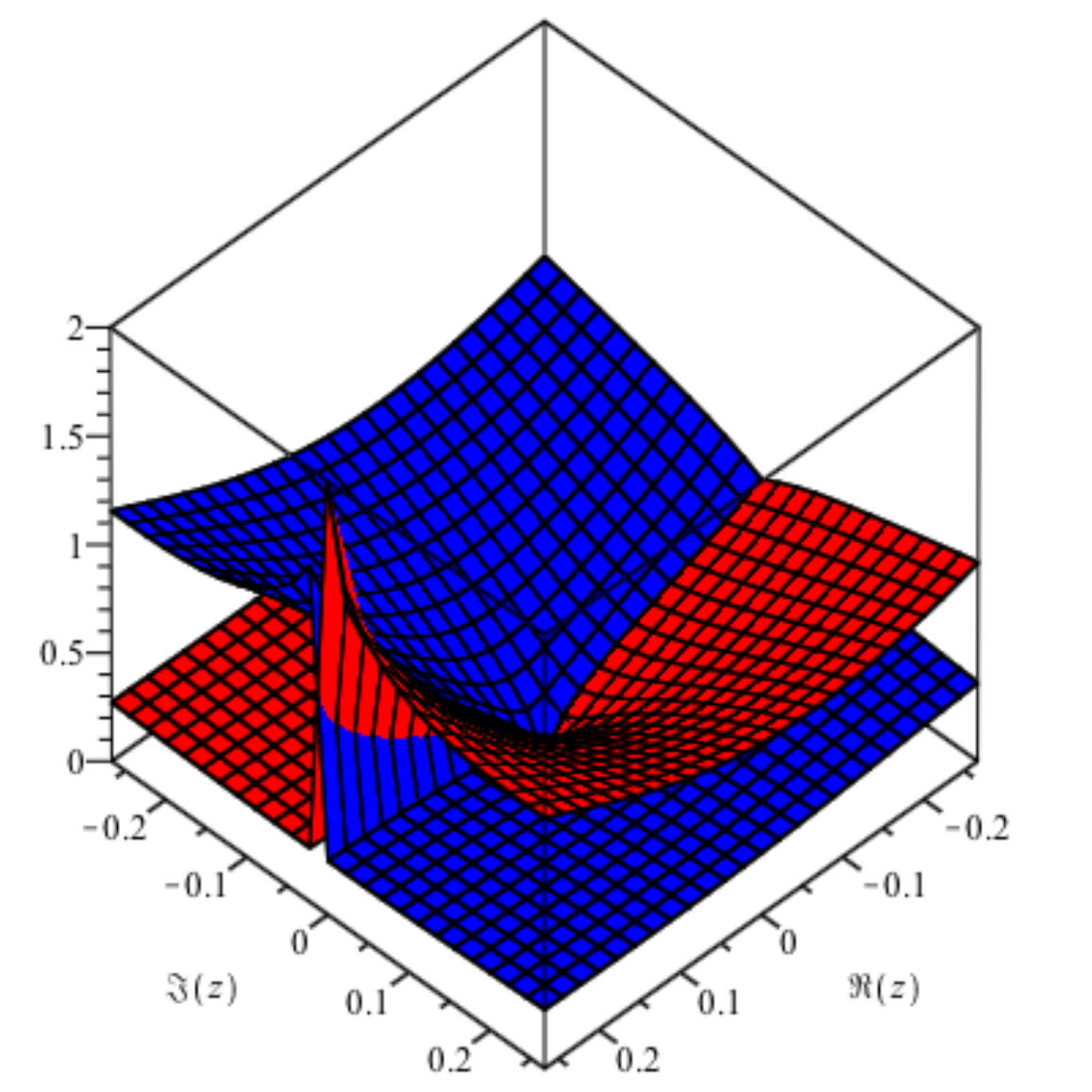}}
\put(-25,5){\tiny\cred $|\widetilde{u}_2(re^{i\theta})|$}
\put(-34 ,-17){\tiny\cblue $|\widetilde{u}_1(re^{i\theta})|$}
\put(-8,0){\tiny$r=0.0001$}
\put(43,0){\tiny$r=0.0001$}
\put(-40,20){\scriptsize $P(u)=u+\dfrac{3}{u}+\dfrac{1}{u^2}\qquad
  1-zP({\cblue u_j(z)})=0\quad 
\widetilde{u}_j(z)
=\sum_{n=0}^{100}\dfrac{z^n}{n!}\left.\dfrac{d^nu_j(z)}{d^nz}\right|_{z=0}
\quad (j\in\{1,2\})$}
\put(15,-20){\tiny\cblue $|u_1(re^{i\theta})|$}
\put(43,-20){\tiny\cred $|u_2(re^{i\theta})|$}
\put(2,-29){\tiny$\theta$}
\put(52,-29){\tiny$\theta$}
\put(60,8){\tiny\cblue$|u_1(z)|$}
\put(60,5){\tiny\cred$|u_2(z)|$}
\put(80,12){\tiny$z\in [[-\rho(1+i),+\rho(1+i)]]$}
\put(90,8){\tiny $\rho=\dfrac{4}{15}$}
\put(-40,-60){\rule{\textwidth}{0.2mm}}
\end{picture}
\ \vspace{6cm}
\caption{\label{fig:stokes}(1-Left) Stokes phenomenon on the truncated
  series of $u_1(z)$ and $u_2(z)$. (2-Center) The correct behaviour.
(3-Right) A intermixed view of the absolute values of the small roots
  that contradicts the domination property stated in \cite{BaNi2010}. The
  example given here is taken from Wallner~\cite{MichaelWallner2018}}
\end{figure}

We will consider
later the domination property in the periodic case.

\subsection{Singularities of an aperiodic walk}
\label{sec:asing}
The discriminant\footnote{See Flajolet-Sedgewick
book~\cite{FlajoletSedgewick2009} p.495.} $R(z)$ of the kernel 
$u^cK(z,u)=u^c(1-zP(u))=0$ with $u$ as the main variable
provides the singularities $\zeta_k$ of its roots as
\begin{equation}
\label{eq:resul}
\zeta_k =\frac{1}{P(\upsilon_k)}\qquad\text{ with  }\qquad P'(\upsilon_k)=0.
\end{equation}
The real point $\rho=1/P(\tau)$ with $P'(\tau)=0$ and  $\tau\in\bR^+$ is a
singularity. We prove next that there are no other singularities within the disk $|z|\leq \rho$.

The following example shows that the expansion at $z=0$ of the dominant real small
root $u_1(z)$ of the kernel $K(z,u)$ 
 has not always  positive coefficient
and we cannot therefore make
use directly on
the roots $u_1(z)$ and $v_1(z)$ of Pringsheim's Theorem  
(see Flajolet-Sedgewick book~\cite{FlajoletSedgewick2009} p. 240) 
which supposes expansions at zero with non-negative coefficients. 
\begin{example}
\label{ex:negcoef}
\begin{eqnarray*}
&&\text{Let }  P(u)=\frac{17}{24}u+\frac{1}{6u^2}+\frac{1}{8u^3},\\
&&\text{We have }
P(1)=1,\  P'(1)=0, \quad\text{and}\quad u_1(z)
     =\frac{z^{1/3}}{2}+\frac{z^{2/3}}{9}-\frac{4z^{4/3}}{2187}+O(z^{5/3}).
\end{eqnarray*}
With an expansion at 10 digits, we obtain for the set $\Xi$ of singularities of 
$1-zP(u)$ and $i=\sqrt{-1}$
\[\Xi\approx\{1, -1.927703811, -0.2861480946 + 1.107549741 i, -0.2861480946
- 1.107549741 i\}.\]
\end{example}
For $u>0$ and $z>0$, since $f_{n,j}\geq 0$, the positive function
 ${\displaystyle F^{[> h]}(z,u)=\sum_{\substack{n\geq 0\\ -cn\leq j\leq
      dn}}f_{n,j}u^jz^n}$ 
is dominated term by term by the positive function ${\displaystyle G(z,u)=\sum_{\substack{n \geq 0\\-cn\leq j \leq dn}}g_{n,j}=\dfrac{1}{1-zP(u)}}$;
The function $G(z,u)$ refers to the set of unrestricted walks while $F^{[>h]}(z,u)$ is
  a subset of the latter, the set of walks with heights greater than
  $h$; the function $B_h(z)=[u^0]F^{[>h]}(z,u)=\sum_{n\geq 0}b_n^{>\infty} z^n$
 refers to the set of bridges,  a
  subset of both previously mentioned sets  of walks. Therefore
\[
 f_{n,k}\geq 0, \ g_{n,k}\geq 0 \Rightarrow F^{[>h]}(z,u)\triangleleft G(z,u),\quad  f_{n,k}\leq g_{n,k},
\quad \  b_n^{>h} = f_{n,0} < \sum_{-cn \leq j \leq dn}f_{n,k}.
\]
The series $G(z,u)$ seen as a function of $z$ is convergent if
$|z|<1/|P(\tau)|=\rho$ and
divergent on the contrary. 

Pringsheim's Theorem~\cite{FlajoletSedgewick2009} states
that $G(z,u)$ has a singularity at  $z=\rho$.

The Laurent polynomial $P'(u)$ cannot have roots $\upsilon$ with
$|\upsilon|<\rho$, which could contradict the preceding facts.

The development at the origin of $B_h(z)$ has non negative
coefficients and the singularity of $B_h(z)$ can only come from the
singularities of the roots $u_i(z)$ or $v_j(z)$ or of cancellations of 
terms $v_m-v_k$ in $Q_k(v_k)$ in Equation (\ref{eq:Piofi}); however
$v_m(z)=v_k(z)$ occurs only at singularities $\zeta=1/P(\upsilon)$ verifying
Equation (\ref{eq:resul}) with $P'(\upsilon)=0$, which is only
possible for $|\upsilon|\geq\rho$.

$B_h(z)$ is  dominated by $G(z,u)$; therefore $B_h(z)$ has radius of convergence
$\rho'\leq\rho$. Since the small root $u_1(z)$ has a singular point at
$z=\rho$, its radius of convergence is $\rho$.

Lemma~\ref{lem:domin} of Banderier-Flajolet~\cite{BanderierFlajolet2002} insures by the triangle inequality 
that for an aperiodic walk $z=\rho$ is the
lone singularity on the circle $|z|=\rho$, corresponding to the root
$u=\tau$ of $P'(u)=0$.

We summarize  this section by the following property.
\begin{property}
\label{prop:domin}
The roots of the kernel equation  $K(z,u)=u^c(1-zP(u))=0$ of an aperiodic walk 
\begin{itemize}
\item have no singularity within the punctured disk 
 $|z|\leq\rho=\dfrac{1}{P(\tau)}\setminus \{z=\rho\}$;
\item the dominant large $v_1(z)$ and small $u_1(z)$ roots of the
  kernel equation have a singularity at $z=\rho$.
 \end{itemize}
\end{property}

\begin{figure}[t!]
\begin{center}
\begin{tikzpicture}[scale=2]
\path(0,-1.5) coordinate (Y1);
\path(0,1.5) coordinate (Y2);
\path(-2,0) coordinate (X1);
\path(-1,1.8) coordinate (R);
\path(4,0) coordinate (X2);
\path(0,0) coordinate (origin);
\path(1,0) coordinate (P1);
\path(0.7071,0.7071) coordinate (P2);
\path(0.97,0.2456) coordinate (rplus);
\path(0.97,-0.2456) coordinate (rminus);
\path(3.0,0.2456) coordinate (Rplus);
\path(2.0,-1.3) coordinate (Ccontour);
\path(2.0,0.2426) coordinate (Rplusmean);
\path(3.0,-0.2456) coordinate (Rminus);
\path(2.0,-0.2456) coordinate (Rminusmean);
\path(3.5,-0.1) coordinate (tauminus);
\path(3.5,0.1) coordinate (tauplus);
\path(3.5,0.18) coordinate (tautau);
\path(-0.15,-0.15) coordinate (OO);
\path(-1.3,-1.3) coordinate(GAMMA);
\path(2,2.4) coordinate (rtozero);
\path(-1.3,2) coordinate (gamma);
\path(-2,2.2) coordinate (stozero);
\begin{scope}[line width=0.1mm,color=black]

\draw[->] (Y1) -- (Y2);
\draw[->] (X1) -- (X2);
\draw[->] (origin) -- node[left=1pt] {$r$} (P2);
\draw[->,green,thick] (rplus) -- (Rplus);
\draw[->,green,thick] (Rminus) -- (rminus);
\draw[->,red] (Rplus) -- (Rminus);
\draw[<->] (Rminusmean) -- node[right=2pt,fill=white] {$2s$}
(Rplusmean);
\draw[-] (tauminus) -- (tauplus);
\node at (tautau) {\large $\rho$};
\node at (OO) {\large $O$};
\node at (GAMMA) {\large $\Gamma_{r}=\overrightarrow{R^-R^+}$};
\node at (gamma) {\large $\begin{array}{l}
               \gamma^+=\overrightarrow{R^+S^+}\\
               \gamma^{\perp}=\overrightarrow{S^+S^-}\\
               \gamma^-=\overrightarrow{S^-R^-}
               \end{array}$};
\node at (Ccontour) {\large $\cC_{r}=\gamma^+\gamma^{\perp}\gamma^{-}\Gamma_{r}$};
\node at (rtozero) {\large $\begin{array}{l}r\rightarrow 0^+,\quad s=r^2\\ 
                                        x=r\cos\arcsin(s/r)\\
                                        R^+: (x,s)\\
                                        R^-: (x,-s)\\
                                        S^+: (y,s),\quad y\in[x,\rho[\\
                                        S^-: (y,-s)\\
                                        P'(\tau)=0,\quad \rho=1/P(\tau)
                           \end{array}$};
\end{scope}
\fill[red] (origin) circle(0.05);

\draw[red] (0.97,0.2431) arc (14.07:345.93:1cm);
\fill[blue] (rplus) circle(0.05);
\fill[blue] (rminus) circle(0.05);
\node at (1.2,0.40) {\large$R^+$};
\node at (1.1,-0.40) {\large$R^-$};
\fill[blue] (Rplus) circle(0.05);
\fill[blue] (Rminus) circle(0.05);
\node at (1.2,0.40) {\large$R^+$};
\node at (1.1,-0.40) {\large$R^-$};
\node at (3.0,0.40) {\large$S^+$};
\node at (3.0,-0.40) {\large$S^-$};
\node at (0,-2.3) {${\displaystyle\left.\begin{array}{l}
           r\rightarrow 0^+\\
           s=o(r^{4r})
          \end{array}\right| \Longrightarrow \left\{
          \int_{\Gamma_{r}}\frac{[u^0]F^{|\geq
        h}(z,u)}{z^{n+1}}dz=o(r^n),\quad\int_{\gamma^{\perp}}\frac{[u^0]F^{|\geq
        h}(z,u)}{z^{n+1}}dz=O(s)=o(r^n)\right\}}$};
\end{tikzpicture}
\end{center}
\begin{picture}(0,0)(0,0)
\put(0,-35){\rule{\textwidth}{0.2mm}}
\end{picture}
\caption{\label{fig:contour}The asymptotic  simplifications 
(see Lemma~\ref{lem:contourrho})}
\end{figure}

\subsection{Asymptotic simplications}
\label{sec:asympsimp}
Equation (\ref{eq:overB}) gives for $B_h(z)=z \overline{B}(z)$
\begin{equation}
\overline{B}_h(z)=\sum_{k=1}^d\sum_{j=1}^c \left(\frac{u_j}{v_k}\right)^h \frac{Q_k(u_j)}{Q_k(v_k)}\frac{u'_j}{v_k}
\end{equation}
Banderier-Nicodeme~\cite{BaNi2010} apply inside the domain
$\widehat{\mathcal{D}}$ verifying
$\widehat{\mathcal{D}}=\{z,|z|<\rho\}$
the following asymptotic simplifications for $j>1, k>1$ and
$h=\Theta(\sqrt{n})$:
\begin{equation}
\label{eq:simpbani}
\left.\begin{array}{l}
\left(\dfrac{u}{v_j}\right)^h =\left(\dfrac{u}{v_1}\right)^h\times\left(\dfrac{v_1}{v_j}\right)^h
 = \left(\dfrac{u}{v_1}\right)^h\times O(\widehat{A}^n),\\[3ex]
u_k^h = u_1^h\left(\dfrac{u_k}{u_1}\right)^h
 = u_1^h\times O(\widehat{B}^{\,n})\end{array}\!\!\!\!\\\\\Longrightarrow\right|
  \ \overline{B}_h(z)=\left(\frac{u_1(z)}{v_1(z)}\right)^h\frac{Q_1(u_1)}{Q_1(v_1)}
\times(1+O(\widehat{C}^{\,n}))
\end{equation}
where $\widehat{C}=\max(\widehat{A},\widehat{B})$ with
 $\widehat{A}:={\displaystyle \max_{2\leq j\leq d} \max_{|z|<\rho}
\frac{|v_1(z)|}{|v_j(z)|}}$ and  $\widehat{B}:={\displaystyle \max_{2\leq k\leq c} \max_{|z|<
 \rho} \frac{|u_k(z)|}{|u_1(z)|}}$
while $\widehat{A}<1$ and $\widehat{B}<1$ by the domination property of Lemma~\ref{lem:domin}.

However the domination properties 
cannot be extended to the disk $|z|<\rho$, as observed by 
Wallner~\cite{MichaelWallner2018}. Figure~\ref{fig:stokes} (Center and
Right) exhibits a
counter-example
for $P(u)=u+\frac{3}{u}+\frac{1}{u^2}$. For $r=0000.1$ we have
$|u_1(re^{it})|>|u_2(re^{it})|$ if $t\in ]0,\pi[$, but the reverse
    occurs when $t\in ]\pi,2\pi[$.
\begin{remark}
\label{rem:domi}
We design in Figure~\ref{fig:contour} a contour on which we will apply the
domination property only on a small neighborhood $\mathcal{D}$ of the real segment $]0,\rho[$,
\[ \mathcal{D}=\{z\pm is\}, \quad\text{with}\quad z\in
  ]0,\rho[\text{ and } s\rightarrow 0,\]
 over which, by continuity,  this property is valid.
\end{remark}
We prove in this section the following.
\begin{lemma}
\label{lem:contourrho}
The integrals of $B_h(z)$ along the
path $\gamma^{\perp}$ and $\Gamma_{r}$ of Figure~\ref{fig:contour}
verify as $r\rightarrow 0$ and $s=o(r^{3n})$
\begin{equation}
\label{eq:u0Fgeqh}
{\it(i)}\ 
 \mathcal{I}_{\perp}=\int_{\gamma^{\perp}}\frac{B_h(z)}{z^{n+1}}dz=O(s)=o(r^n),
\quad(\text{ii})\  \mathcal{I}_r=\int_{\Gamma_{r}}\frac{B_h(z)}{z^{n+1}}dz=o(r^n), 
\end{equation}
and therefore, using Remark~\ref{rem:domi},
\begin{equation}
\label{eq:limcauchy}
\frac{1}{2\pi i}\int_{\mathcal{C}_{r}}\frac{B_h(z)}{z^{n+1}}dz
=\frac{1}{2\pi i}\int_{\gamma^+\cup\gamma^-}\frac{B_h(z)}{z^{n+1}}dz+o(r^n).
\end{equation}
\end{lemma}
\begin{proof}
As $r\rightarrow 0^+$ and $s=o(r^{3n})$, the path
$\overrightarrow{R^-R^+}$ has for limit the quasi-circle 
$\mathcal{C}_{r}=\left\{z=r
  e^{i\nu}; \ \nu\in [r,2\pi-r]\right\}.$
We will show that, although surprising at first sight, the integration
along this path has an exponentially small and negligible contribution
to the end result.

Along the segments $R^+S^+$ and $S^-R^-$ the domination property of
the large and small roots of the kernel apply by continuity as $s\rightarrow 0^+$.

The abscissa $y$ of $S^+$ and $S^-$ has been chosen strictly less than
$\rho$, the abscissa of
the critical point; this implies, as $r$ tends to zero, that along the
segment $S^+S^-$ all
the large and small roots and their first 
derivatives are finite\footnote{Following Banderier-Flajolet~\cite{BanderierFlajolet2002}, by differentiating
  $1-zP(u(z))$, we obtain for each branch
  $u'_j(z)=z^2P^{-1}(u_j(z))$. Using the duality
  $\widetilde{P}(u)=P(1/u)$, we obtain a similar property for the
  large roots $v_j$.}. 
Therefore, since  the integrand is finite along this segment
and 
$|S^+S^-|=2s\rightarrow 0^+$, the value of the integral along this
segment is
$o(s)=o(r^{3n})$ as $r\rightarrow 0^+$.
This proves Part (\textit{i}) of the Lemma.

Since $Q_k(u)=\prod_{2\leq m\leq d, m\neq k}(u-v_m(z))$, we have
\begin{equation}
\label{eqovQu}
\frac{Q_k(u_j)}{Q_k(v_{k})}=\frac{\prod_{m=1,m\neq k}^d(u_j-v_m)}{\prod_{m=1,m\neq k}^d (v_{m}-v_k)}
      =\frac{\prod_{m=1}^d(u_j-v_m)}{(u_j-v_k)v_{k}^{d-1}\prod_{m=1,m\neq
          k}^d\left(1-\dfrac{v_m}{v_{k}}\right)}.
\end{equation}
We decompose $\overline{B}_h(z)$ as a sum of products, with 
$\xi:=x\sigma=x\sqrt{P''(\tau)}$ for $x\in]0,\infty[$ and $h=\xi
                                          \sqrt{n}$.
We consider in this section the formal case where the height $h$ is 
any real positive number; the ``combinatorial'' case where $h$ is integer is embedded
in the latter. We refer to Section~\ref{sec:hinteger} of the periodic case  for a
proof
when  $h$ is integer. 
\begin{align}
\label{eq:sBh}
&\frac{\overline{B}_h(z)}{z^{n}}=\frac{1}{z^n}\sum_{k=1}^d\sum_{j=1}^c
            A_{jk}B_{jk}C_{k}D_{jk},\qquad\text{ where }\\
& A_{jk}=\left(\frac{u_j}{v_k}\right)^h, \ 
B_{jk}=\frac{\prod_{m=1}^d(u_j-v_m)}{(u_j-v_k)v_{k}^{d-1}},
\ C_{k}=\!\!\!\!\!\prod_{m=1,m\neq
          k}^d\!\! \left(1-\dfrac{v_m}{v_{k}}\right),\ D_{jk}=\frac{u'_j}{v_k}.
\end{align}
Using Equation (\ref{eq:sumjk}), we obtain upon the contour $\Gamma_{r}=\overrightarrow{R^-R^+}$ a sum of
$dc$ integrals of the type
\begin{equation}
\label{eq_jjprim}
I_{jk}=\frac{1}{2\pi i}\int_{\Gamma_{r}}\frac{1}{z^n}
A_{jk}(z)B_{k}(z)C_{k}(z)D_{jk}(z)dz,\qquad
\mathcal{I}_r=\sum_{j=1}^d\sum_{k=1}^c I_{jk}.
\end{equation}
We expand  each term of the integrand of  $I_{jk}$ in a neighborhood of $z=0$, show that
$B_{k}$ and $C_{k}$ are constants up to negligible terms, and
combine the asymptotics obtained.

We need to define the notations of the second order terms $z^r$ or $z^s$ in the
asymptotic expansions of $\overline{B}_h(z)$ at $z=0$.
\begin{definition}
\label{def:secorder}
Let $\widetilde{r}$ (resp. $\widetilde{s}$) be the 
degree of
the dominant monomial of $P(u)-p_du^d$
$\left(\text{resp. } P(u)-\dfrac{p_{-c}}{u^c}\right)$ as
$u\rightarrow +\infty$ (resp. $u\rightarrow 0$),
and
$r=\max(\widetilde{r},0),\  s=\max(-\widetilde{s},0)$.
\end{definition} 
\begin{example}
\begin{align*}
& P(u)=u^3+u+\frac{1}{u^2}+\frac{1}{u^5},&& r=\widetilde{r}=1, && s=-\widetilde{s}=2,\\
& P(u)=u^3+\frac{1}{u}+\frac{1}{u^2},&& \widetilde{r}=-1,\ r=0,&& s=-\widetilde{s}=1,\\
&P(u)= u^d+\frac{1}{u^c},&&\widetilde{r}=-c,\ r=0, &&-\widetilde{s}=-d,\ s=0.
\end{align*}
\end{example}
With the $c_j$ constants independent of $z$, asymptotics as $z$
tends to zero at first or second order by boot-strapping provides:
\begin{align*}
&A_{jk}=\left(\frac{u_j}{v_k}\right)^{\xi\sqrt{n}}=\frac{\omega_j}{\varpi_k} z^{\xi\sqrt{n}\left(\frac{1}{c} +\frac{1}{d}\right)}
\left(1+O\left(z^{1+\frac{1}{d}+\frac{1}{c}
  -\max(\frac{r}{d},\frac{s}{c})}\right)\right)
\\
&  B_{jk}=\frac{\prod_{m=1}^d(u_j-v_m)}{(u_j-v_k)v_k^{d-1}}=1+O\left(z^{1-\frac{r}{d}}\right),\ 
\text{since } v_m-u_j \sim v_m=c_1
  z^{-1/d}\times\left(1+O\left(z^{1-\frac{r}{d}}\right)\right)    \\
&C_{k}=\prod_{m=1,m\neq
    k}^d\left(1-\frac{v_m}{v_k}\right)=\prod_{m=1,m\neq
    k}(1-v_{m-k})=\prod_{1\leq m \leq d-1}(1-v_{m})=
      d\left(1+O\left(z^{1-\frac{r}{d}}\right)\right)\\
&   D_{jk}=\frac{u'_j}{v_{k}}
 =\frac{1}{c}\frac{\omega_j}{\varpi_k}z^{-1+\frac{1}{c} +\frac{1}{d}}\left(1+O\left(z^{1 -\max(\frac{s}{c},\frac{r}{d})}\right)\right)
\ \text{since }\left\{\begin{array}{l}
    u_j=\omega_j z^{\frac{1}{c} }\left(1+O\left(z^{1-\frac{s}{c}}\right)\right)\\
    u'_j=\frac{1}{c}\omega_jz^{-1+\frac{1}{c} }\left(1+O\left(z^{1-\frac{s}{c} }\right)\right)
 \end{array}\right.
 \end{align*}
Collecting the preceding expansions, we get with $c_{jk}$ a constant
\begin{align}
\label{eq:Isimp} I_{jk}&
   =\frac{c_{jk}}{2\pi i}\int_{\Gamma_{r}}z^{-n}z^{h}
z^{\alpha}\times\left(1+O\left(z^{\beta}\right)\right)dz,
\quad \left\{\begin{array}{l}
\alpha=-1+\frac{1}{c}+\frac{1}{d},\\[2ex]
\nonumber
\beta=1+\frac{1}{d}+\frac{1}{c}-\max\left(\frac{s}{c},\frac{r}{d}\right).
\end{array} \right.\\
& \text{where } -1<\alpha<1 \quad\text{and}\quad 0<\beta<3
\end{align}
To compute ${\displaystyle
  J_{jk}=\frac{c_{jk}}{2i\pi}\oint_{\Gamma_r}z^{-n}z^hz^{\alpha}dz}$ we make the ubiquitous changes
 of variable
\begin{align}
\nonumber
&z=r\left(1-\dfrac{t}{n}\right)\qquad\text{to get an expansion for
    large  } n\\
&z=r e^{i\nu}\label{eq:znu}
\end{align}
We could integrate directly $J_{jk}$ as a function of $\nu$ after the
change of variable $z\leadsto r e^{i\nu}$
 along the path $\Gamma_{r}$,
but terms of the form $\exp(N2i\nu)$, with N a large non
integer number,  have a wild behaviour that
is useless for our needs.
 
Neglecting second order terms, both changes of variables lead to
\[t(\nu):=t= \left(1-e^{i\nu}\right)n\qquad\nu\in[s/r,2\pi-s/r] \]
and to an integration along the quasi-circle $\Gamma_{n}$, obtained
from $\Gamma_{r}$ by a shift $+1$, a symmetry with respect to the line $x=1$,  and a homothety of value $n$. The
resulting contour is centered at
$+1$ and has radius $n$. We remark that $t(0)=t(2\pi)=0$.

We use the standard asymptotic scale for convergence of a discrete walk to a Brownian
motion, which provides for the height $h$,  
\[h=x\sigma\sqrt{n}, \quad \text{with } 
 \left\{\begin{array}{l}
\sigma=\sqrt{P"(1)}\\
 \sigma\text{ standard deviation of the set of the jumps}
 \end{array}\right.
\]
The expansions for large $n$ of $(z/r)^{-n}, (z/r)^h, (z/r)^{\alpha}$ respectively are
\begin{eqnarray}
\label{eq:expzn}
&&\left(1-\frac{t}{n}\right)^{-n}=\exp(t)\left(\sum_{\ell \geq
  0}\frac{E_\ell (t)}{n^\ell }\right)\\
\label{eq:expsigma}
&&\left(1-\frac{t}{n}\right)^{\xi
  \sqrt{n}}=\sum_{\ell \geq 0}\frac{\Xi_\ell (t)}{n^{\ell /2}}\qquad
(\xi=x\sigma)\\
\label{eq:expalpha}
&&\left(1-\frac{t}{n}\right)^{\alpha}
  =\sum_{\ell \geq 0}(-1)^\ell  
              \frac{t^\ell }{\ell ! n^\ell }\frac{\Gamma(\alpha+\ell )}{\Gamma(\alpha)}\\[1ex]
&&\text{where }  E_\ell (t) \text{ and } \Xi_\ell (t) \text{ are polynomials of
                degree at most } 2\ell 
\end{eqnarray}
Collecting these asymptotics, we  identify $s/r$ and
$\arcsin(s/r)$ as $s/r \rightarrow 0$.
We set $M=n-\xi\sqrt{n}-\alpha$, and  we obtain with $\alpha_{g}$ and
$\eta_{g,q}$ constants, $j\in\{1,..,c\}$ and $k\in\{1,..d\}$,
\begin{eqnarray}
\label{eq:sumasympt}
&\dfrac{2i\pi}{c_{jk}} \times J_{jk}&
   =\int_{\Gamma_{r}}z^{-n}z^{h}
z^{\alpha}dz=
\sum_{g>0}
  \int_{s/r}^{2\pi-s/r}\!\!\!\!\!r^{-M}\alpha_{g}\frac{1}{n^{g/2}}\sum_{q\leq
    2g}\eta_{g,q}e^{t(\nu)}t^q(\nu)\frac{dt(\nu)}{d\nu}d\nu
  \\
\nonumber
&&=\sum_{g>0}\alpha_{g}\frac{r^{-M}}{n^{g/2}}J_{jk,g},
\quad \text{where}
\quad
J_{jk,g}=\sum_{q\leq
    2g}\eta_{g,q}
\int_{s/r}^{2\pi-s/r}\!\!\!e^{t(\nu)}t^q(\nu)\frac{dt(\nu)}{d\nu}d\nu
\end{eqnarray}
We prove next that $J_{jk,g}=o(r^{2n})$ as $s=o(r^{3n})$ and
$r\rightarrow 0$. We integrate the generic term
\[M_k =\int_{s/r}^{2\pi-s/r} e^{t(\nu)}t^k(\nu)dt(\nu).\]
Since $t(\nu)=n(1-e^{i\nu})$, we do the change of variable $t(\nu)= n
s(\nu)$, and we integrate as follows,
\begin{align}
\label{eq:Ek}
\frac{E_k}{n}=\int e^{s(\nu)} s^k(\nu)ds(\nu) 
&=\int
  \left(1-e^{i\nu}\right)^k e^{\displaystyle
    -e^{i\nu}}(-i e^{i\nu})d\nu.\\
\nonumber &= (1-e^{i\nu})^ke^{\displaystyle -e^{i\nu}}
      -\int k(1-e^{i\nu})^{k-1} e^{\displaystyle
        -e^{i\nu}}(-ie^{i\nu})d\nu\\
     &=(1-e^{i\nu})^ke^{\displaystyle -e^{i\nu}}+k E_{k-1}\\
     &=e^{\displaystyle -e^{i\nu}}P_k(e^{i\nu}), 
\end{align} 
where $P_k(x)$ is a polynomial of degree $k$ with minimum degree at
least  $1$
 and  
coefficients bounded by $n^k$.

The periodicity of the
 trigonometric function $e^{i\nu}$ provides for $M_k$
with $s=o(r^{3n})$ and $r\rightarrow 0$
\begin{equation}
\label{eq:limraper}
\left[e^{-e^{i\nu}}\nu^j\right]_{\nu=s/r}^{2\pi-s/r}=-2i(s/r)^j e^{-1}+O((s/r)^{j+1})=o(r^{2n})
  \Longrightarrow\left\{\begin{array}{l}
  \Big[E_k\Big]_{\nu=s/r}^{2\pi-s/r}=o(r^{2n})\\[2ex]
   J_{jk,g}=o(r^{2n}),\end{array}\right.
\end{equation} 
\begin{equation}
\label{eq:limint0}
J_{jk}=o(r^n) \qquad\text{and} \qquad\cJ_{r}= \sum_{j=1}^d\sum_{k=1}^c
J_{jk}=
\sum_{k}c_{jk}\frac{1}{2i\pi}\int_{\Gamma_{r}}z^{-n}z^hz^{\alpha}dz=o(r^n).
\end{equation}
We expand Equation (\ref{eq:Isimp}) to handle the error term,
\begin{equation}
\label{eq:errnoper}
I_{jk}
   =J_{jk}+\frac{1}{2\pi
  i}\int_{\Gamma_{r}}O\left(z^{-n+h+\alpha+\beta}\right)dz,
\end{equation}
with ${\displaystyle
|\alpha|<1\text{ and }  0<\beta<3}$.
We use Theorem VI.9 (Singular integration) of
Flajolet-Sedgewick~\cite{FlajoletSedgewick2009}
which states the following:

{\sl Let $f(z)$ be $\Delta$-analytic and admit an expansion near its
  singularity of the form
\[f(z)=\sum_{j=0}^J c_j (1-z)^{\alpha_j}+O\left((1-z)^A\right).\]
Then $\int_0^z f(t)dt$ is $\Delta$-analytic. Assume that none
of the quantities $\alpha_j$ and $A$ equal $-1$

If $A<1$ the singular expansion of $\int f$ is}
\[\int_0^z f(t)dt =
  -\sum_{j=0}^J\frac{c_j}{\alpha_j+1}(1-z)^{\alpha_j+1}+O\left((1-z)^{A+1}\right).
\]
We apply this theorem to the BigO term of Equation
(\ref{eq:errnoper}) by shifting the origin to any real point,
$z\leadsto z-\alpha$, which gives
\[\mathcal{I}_O=\int O\left(z^{-n+h+\alpha+\beta}\right)dz=
O\left(z^{-n+h+\alpha+\beta+1}\right)\]
Expanding $z^n$, $z^{h}$ and $z^{\alpha+\beta+1}$
as in Equations
(\ref{eq:expzn},\ref{eq:expsigma},\ref{eq:expalpha}), and making the
developments that follow until Equation (\ref{eq:limraper}) leads to
\begin{equation}
\label{eq:limomega}
\Big[\mathcal{I}_O\Big]_{s}^{2\pi-s}=o(r^n) \quad \text{as} \quad
 n\rightarrow \infty,\quad s=o(r^{4n}),\text{ and } r\rightarrow 0.
\end{equation}
When the contour $\Gamma_{r}$ is shrunk to zero, 
we have therefore
 $I_{jk}=o(r^n)$
where $I_{jk}$ has been defined in Equation (\ref{eq:Isimp}).
\end{proof}

\subsubsection{Using the domination property}
\label{sec:ssasympaper}
Lemma~\ref{lem:contourrho} gives
us\footnote{Banderier-Nicodeme~\cite{BaNi2010} provide
Sections~\ref{sec:ssasympaper} and~\ref{sec:semilarge} when
 $\tau=\rho=1$.}, 
\begin{equation}
\label{eq:intgamma}
\frac{1}{2\pi i}\oint_{\mathcal{C}_{r}}\frac{B_h(z)}{z^{n+1}}dz
=\frac{1}{2\pi i}\left[\int_{\gamma_+}+\int_{\gamma_-}\right]\frac{B_h(z)}{z^{n+1}}dz+o(r^n),\quad
B_h(z)=[u^0]F^{[>h]}(z,u).
\end{equation}
As $s$ tends to zero, 
on the segments of integration $\gamma_+$ and $\gamma_{-}$, the
domination property of the roots of the kernel applies, namely,
\begin{equation}
|u_j(z)|< u_1(z) < v_1(z) < |v_k(z)|, \quad \left\{\begin{array}{l}
           u_1(z) \text{ dominant small kernel root}\\ 
           v_1(z) \text{ dominant large kernel root}\\
           j\neq 1,\ k\neq 1
           \end{array}\right.
\end{equation}
Since, as $s\rightarrow 0^+$, along $\gamma^+$ and $\gamma^-$,  we have
\[\left(\frac{u_j(z)}{v_k(z)}\right)^h=O(A^h), \quad A=\max\left|\frac{u_j}{v_k}\right|<1\text{ for }  j\neq 1 
\text{ or } k\neq 1\text{ and }z\in[r\cos(s),\rho[,\]
and therefore, Equation (\ref{eq:overB}) verifies with $A<1$
\begin{equation}
\label{eq:Bhsimp}
\overline{B}_h(z)=\sum_{k=1}^d\sum_{j=1}^c \left(\frac{u_j}{v_k}\right)^h
\frac{Q_k(u_j)}{Q_k(v_k)}\frac{u'_j}{v_k}=\left(\frac{u_1(z)}{v_1(z)}\right)^h\frac{Q_1(u_1(z))}{Q_1(v_1(z))}\frac{u'_1(z)}{v_1(z)}+O(A^h)
\end{equation}
We are in the domain of semi-large powers with $h=\Theta(\sqrt{n})$
(see~\cite{FlajoletSedgewick2009} Section IX.11.2), and the dominant
asymptotic terms comes from the dominant singularity of $B_h(z)$
located a $z=\rho$.

We follow Banderier-Flajolet~\cite{BanderierFlajolet2002} and  expand $\dfrac{1}{P(u)}$ in
 the neighborhood of $u=\tau$, the  exceptional point
 corresponding to the lone singular point $z=\rho=\dfrac{1}{P(\tau)}$
 of the kernel equation on the circle $|z|=\rho$, and invert next
 $z-\dfrac{1}{P(u)}=0$ as a function $u(z)$.
We observe that $P''(\tau)>0$ for $u\in\bR^+$ and that the non dominant roots $u_j(z)$
with $j>1$ and $v_k(z)$ with $k>1$ are regular at $z=\rho$.
\begin{align}
\nonumber
&z=\frac{1}{P(u)}=\frac{1}{P(\tau)}-\frac{P''(\tau)(u-\tau)^2}{2P(\tau)}+O((u-\tau)^3)\\
& \Longrightarrow\left\{\begin{array}{l}
       u_1(z)=\tau-\dfrac{\sqrt{2\rho}}{\sigma}\sqrt{1-z/\rho}+O(1-z/\rho),\\[2ex]
       v_1(z)=\tau+\dfrac{\sqrt{2\rho}}{\sigma}\sqrt{1-z/\rho}+O(1-z/\rho),\\[2ex]
       \dfrac{u'_1(z)}{v_1(z)}=\dfrac{1}{\sqrt{2}\sigma\rho^{3/2}\tau\sqrt{1-z/\rho}}
           \times\left(1+O\left(\sqrt{1-z/\rho}\right)\right)
   \end{array}\right.\qquad \begin{array}{l}
\rho=\dfrac{1}{P(\tau)},\\[3ex] \sigma=\sqrt{P''(\tau)}
\end{array}
\nonumber
\end{align}
Since ${\displaystyle Q_1(u)=\prod_{2\leq k \leq d}(u-v_k(z))=\sum_{m=0}^{d-1}q_m(z)u^m}$,
we obtain\footnote{Taking expansions of $u_1(z)$,
$v_1(z)$ and $Q_1(u_1(z))/Q(v_1(z))$ at $z=\rho$  at higher order would produce
a real series where the coefficient of terms like $(1-z/\rho)^{k/2}$
are symmetric functions of $u_2(\rho),\dots,u_c(\rho),v_2(\rho),\dots,v_d(\rho)$.} at order 1
\begin{equation}
\label{eq:QoverQ}
\frac{Q_1(u_1(z))}{Q_1(v_1(z))}=\frac{Q_1(\tau)+O(\sqrt{1-z/\rho})}{Q_1(\tau)+O(\sqrt{1-z/\rho})}
       = 1+O(\sqrt{1-z/\rho}) \quad \text{ as }z\sim \rho^-.
\end{equation}
On the other hand,
\begin{equation}\label{eq:uvh}
\left( \frac{u_1}{v_1}\right)^h
 = \left(1- \frac{2\sqrt{2\rho}}{\tau\sigma}
 \sqrt{1-z/\rho}\right)^h\times\left(1+O\left(\sqrt{1-z/\rho}\right)\right)
 \qquad \text{for} \ z\sim \rho^{-},
\end{equation}
Collecting the expansions in the neighborhood of $z=\rho$, we get
\begin{equation}
\label{eq:collectrho}
B_h(z)=F_0^{>h}(z)=\left(1-\frac{2\sqrt{2}\sqrt{1-z/\rho}}{\tau\sigma\sqrt{\rho}}\right)^h\frac{1}{\sigma\tau\rho^{3/2}\sqrt{2}\sqrt{1-z/\rho}}\times\left(1+O(\sqrt{1-z/\rho})\right).
\end{equation}

\subsection{Semi-large powers and Hankel integrations}
\label{sec:semilarge}
\noindent
We compute now asymptotically $b_n^{>h}=[z^n]F_0^{[>h]}(z)$ for large
$n$  when $h= x
\sigma\sqrt{n}$ and $x\in\bR^+$, 
the convergence range to the Brownian limit.
By the usual process of singular analysis (see Flajolet-Sedgewick
book~\cite{FlajoletSedgewick2009} Theorem VI.3 and VI.5 - Transfers
and
 Multiple Singularities),
we deform the contour $\mathcal{C}_r$ to a $\Delta$-contour
$\Gamma_{\Delta}$ consisting of set of Hankel 
contours\footnote{\label{foot:zeta}It may occur that some singularities $\widetilde{\zeta}_{ij}$ are
located on the semi-infinite ray $\mathcal{R}_i=\zeta_i\infty_i$ with direction $O\zeta_i$. 
These singularities are ``swallowed'' by the Hankel integral $\mathcal{H}_i$, since an algebraic function is analytic apart on its
singularities and therefore remains analytic at points not belonging
to the ray $\mathcal{R}_i$, whatever close to this ray.
} $\mathcal{H}_i$,
each of which winding
around a  singular point $z=\zeta_i$, and connecting paths at infinity
the
contribution of which  is zero. We observe that
$\overline{B}_h(z)$ is analytic within the contour $\Gamma_{\Delta}$. 
One of the Hankel contours is the dominant one, winding around the dominant
singularity
$z=\rho$.
The singular points $\zeta_i$ are of order 2, and the secondary Hankel integrals
$\mathcal{H}_i$ with $i>1$ and the dominant one $\mathcal{H}_0$
are computed similarly; the former  ones
provide
exponentially small contribution with respect to the latter.
\begin{example}
\label{ex:Hj}
Taking $P(u)$ of  Example~\ref{ex:negcoef}, we have
\[
P(u)=\frac{17}{24}u+\frac{1}{6u^2}+\frac{1}{8u^3},
  \quad P'(u)=\frac{17}{24}-\frac{1}{3u^2}-\frac{3}{8u^4};\]
The roots $\tau_i$ of $P'(u)=0$ and the singular points
$\zeta_i=1/P(\tau_i)$ verify with
\small
\[A=\frac{(17918+5202\sqrt{19})^{1/3}}{51},\quad
B=\frac{17}{3(17918+5202\sqrt{19})^{1/3}},\quad I=e^{i\pi /2},\]
\ \\[-6ex]
\begin{align*}
\tau_0&=1, &\zeta_0&=\rho=1\\[-1ex]
\tau_1&=-\frac{1}{3}-2\left(\frac{A}{2}-B\right),&\zeta_1 &\approx -1.927703810,\\
\tau_2&=-\frac{1}{3}+\frac{A}{2}-B-I\frac{\sqrt{3}}{2}(A+2B),&\zeta_2
&\approx -0.2861480946 + 1.107549741I,\\
\tau_3&=-\frac{1}{3}+\frac{A}{2}-B+I\frac{\sqrt{3}}{2}(A+2B),&\zeta_3
&\approx -0.2861480946 - 1.107549741I.\\
&&&
\end{align*}
\normalsize
\end{example}
\noindent
We have then with $b=c+d$ the number of roots $u(z)$ of the kernel 
equation (see footnote~\ref{foot:zeta})
\begin{equation}
\label{eq:finnoper}
b_n^{>h}=\frac{1}{2i\pi}\oint_{\mathcal{C}_r}\frac{B_h(z)}{z^{n+1}}dz=\mathcal{I}_0+\sum_{j=1}^{e}\mathcal{I}_j,\quad\text{where}\quad
\mathcal{I}_j=\frac{1}{2\pi
  i}\oint_{\mathcal{H}_j}\frac{B_h(z)}{z^{n+1}}dz,\quad e\leq b-1,
\end{equation}
and $\mathcal{C}_r$ is the contour defined in Figure~\ref{fig:contour}.

We develop the computation of the dominant Hankel integral
${\displaystyle\mathcal{I}_0=\frac{1}{2\pi
    i}\oint_{\mathcal{H}_0}\!\!\!\frac{\overline{B}_h(z)}{z^{n}}dz}$
by  following  the proof of Banderier {\it et al}~\cite{BaFlScSo2001}
which refers to semi-large powers (see Theorem
IX.16 of~\cite{FlajoletSedgewick2009}).

Using the change of variable $z=\rho\left(1-\dfrac{t}{n}\right)$, taking an
expansion for large $n$ of the integrand $\overline{B}_h(z)/z^n$ of
$\mathcal{I}_0$, we
have
\begin{equation}
\label{eq:raptoh}
\left(\frac{u_1(z)}{v_1(z)}\right)^{x\sigma\sqrt{n}}=e^{-2\sqrt{2}\sqrt{t}x\sqrt{\rho}/\tau}\times\left(1+O\left(\frac{1}{\sqrt{n}}\right)\right),\qquad (\rho=1/P(\tau)),
\end{equation}
and we obtain
\begin{equation}
\label{eq:b_nbeforehankel}
\mathcal{I}_0
=\frac{1}{2\pi i}\oint_{\mathcal{H}_0}\!\!\!\frac{\overline{B}_h(z)}{z^{n}}dz=
\rho^{-n}\frac{1}{2 \pi i}\frac{1}{\sigma\tau\sqrt{\rho}}\oint_{\mathcal{\widetilde{H}}_0}\! 
 \frac{1}{\sqrt{2}\sqrt{n}}\frac{e^t e^{-2\xi\sqrt{2t}}}{\sqrt{t}}\times\!
 \left(\!1\!+\!O\left(\frac{1}{\sqrt{n}}\right)\!\right)dt,
\end{equation}
where $\xi=x\sqrt{\rho}/\tau$ and $\mathcal{\widetilde{H}}_0$ is a Hankel contour winding clockwise from $-\infty$
around the origin.

Theorem~\ref{theo:BaFlaaperunb} states that the number of unbounded
bridges of length $n$ verifies
\begin{equation}
\label{eq:bninf}
b_n^{<\infty}=V_n=\frac{\rho^{-n}}{\sigma\sqrt{2\pi\rho
    n}}\left(1+O\left(\frac{1}{n}\right)\right)\qquad\text{with }
\rho=\frac{1}{P(\tau)},\  \sigma=\sqrt{P''(\tau)}.
\end{equation}
Expanding $e^{-2\xi\sqrt{2t}}$, making the substitution $t \leadsto -t$, integrating term-wise,
and using the Hankel contour representation\footnote{See a proof 
of this representation in Flajolet-Sedgewick~\cite{FlajoletSedgewick2009}, p. 745.} 
for the Gamma function, 
\begin{equation}
\label{eq:GammaHankel}
G(s)=\frac{1}{\pi}\sin(\pi s)\Gamma(1-s)=-\frac{1}{2 i
  \pi}\int_{+\infty}^{(0)}(-t)^{-s}e^{-t}dt,\  \text{for all\ } s \in
\bC,\quad G(-1/2)=\frac{-1}{2\sqrt{\pi}},
\end{equation}
the computation of  $\mathcal{I}_0$ gives 
\begin{align}
\label{eq:I1afterhankel}
\mathcal{I}_0 / b_n^{<\infty}
 & = -\frac{\sqrt{\pi}}{2 \pi i}\sum_{j=0}^{\infty}(-1)^j\frac{2^j (\sqrt{2}\xi)^j}{j!}
 \int_{+\infty}^{(0)}e^{(-t)}(-t)^{(j-1)/2}dt
 \times\left(1 +O\left(\frac{1}{\sqrt{n}}\right)\right)\\
 & = \sum_{k=0}^{\infty}(-1)^k \frac{(\sqrt{2}\xi)^{2k}}{k!}
 \times\left(1 +O\left(\frac{1}{\sqrt{n}}\right)\right)= e^{-2\xi^2} 
 \times\left(1 +O\left(\frac{1}{\sqrt{n}}\right)\right)\\
\label{eq:xrhotau}
& = e^{-2x^2\rho/\tau^2}\times\left(1 +O\left(\frac{1}{\sqrt{n}}\right)\right).
\end{align}
But we also have
$\dfrac{|\mathcal{I}_j|}{\mathcal{I}_0}=\left(\dfrac{\rho}{|\zeta_j|}\right)^n=O(B^n)$
with $B<1$, which leads 
to the following theorem.
\begin{theorem}
\label{theo:rayleighnoper}
For walks with non-periodic sets of jumps and characteristic polynomial, as $n\rightarrow\infty$,  the probability $\beta_n^{>x\sigma\sqrt{n}}$ that a bridge of length $n$
goes upon the barrier $y=x\sigma\sqrt{n}$ follows a Rayleigh limit
law for $x\in ]0,+\infty[$
\begin{equation}
\label{eq:Ibnafterhankel}
\beta_n^{>x\sigma\sqrt{n}}=\frac{b_n^{>h}}{b_n^{<\infty}}= 
  \frac{\mathcal{I}_0}{b_n^{<\infty}}\times(1+O(B^n))=
 e^{-2x^2\rho/\tau^2} 
 \times\left(1 +O\left(\frac{1}{\sqrt{n}}\right)\right)\quad (B<1),
\end{equation}
where $b_n^{<\infty}$ verifies Equation (\ref{eq:bninf}).
\end{theorem}
\begin{remark}
\label{rem:rayleighnoper}
As observed in Banderier-Nicodeme~\cite{BaNi2010}, this result is
independent of the drift  $P'(1)$ of the walk; it is also independent of the standard
deviation $\sigma=\sqrt{P''(\tau)}$.
\end{remark}
As a consequence of the preceding theorem, in the probabilistic
setting where  $P(1)=1$ and with zero drift $P'(1)=0$ implying $\rho=1$, we have 
as in Banderier-Nicodeme~\cite{BaNi2010}.
\begin{theorem}
\label{theo:stringrayleigh}
Considering an i.i.d. integer valued random variable $X_i=P(u)$ with expectation
$\Ex(X_1)=0$ and standard deviation $\sigma=\sqrt{P"(1)}$, where $P(u)$
is a Laurent polynomial defined as in Equation (\ref{eq:charpoly}), we have
for $S_k=\sum_{1\leq i \leq k }X_i$ a
Rayleigh law
\begin{equation}
\label{eq:stringrayleigh}
\lim_{n\rightarrow \infty}\Pr\left(S_n=0,\max _{0\leq k\leq n}S_k> x\times
\sigma\sqrt{n}\right)=e^{-2x^2}\times\left(1+O\left(\frac{1}{\sqrt{n}}\right)\right)
                                        \qquad x\in]0,+\infty[,
\end{equation}
\end{theorem}
\subsection{Points tight to the Brownian}
\label{sec:tight}
\label{rem:strong}
The strong embedding theorem of
Koml\'os-Major-Tusn\'ady~\cite{KoMaTu75} of which 
Chatterjee~\cite{Chatterjee2012} gave a modern approach provides the
following:
\begin{theorem}
\label{theo:strong}
Given  i.i.d. random
variables $\epsilon_i,\epsilon_2,\dots$ such that
$\Ex(\epsilon_1)=0, \Ex(\epsilon_1^2)=1$ and
 $\Ex \exp\theta|\epsilon_1|<\infty$ for some
$\theta>0$,
it is possible to construct a version of $(S_k)_{0\leq k\leq n}$
 with $S_k =\sum_{i=1}^k \epsilon_i$ and  a standard Brownian motion
 $(B_t)_{0\leq t\leq n}$ on the same probability space such that for
 all $x>0$,
\begin{equation}
\label{eq:strBn}
 \Pr\left(\max_{k\leq n}|S_k-B_k|\leq C\log n +x\right)\leq
 Ke^{-\lambda x},
\end{equation}
where $C,K$ and $\lambda$ do not depend on $n$.
\end{theorem}
Let us consider $\theta\in]0,+\infty[$ and the i.i.d variables $Y_i=X_i/\sigma$ where $X_1$ has probability
distribution $P(u)=p_d u^d + p_{d-1}u^{d-1}+\dots+p_{-c}u^{-c}$, with\footnote{As mentioned
 in Banderier-Nicodeme~\cite{BaNi2010} it is possible to move
 the expectation of a discrete variable to $0$ by the method of {\sl
   shifting  the mean}. See Szpankowski's book~\cite{Szpankowski01}.}
$P(u)$ a positive Laurent polynomial, $P'(1)=0$ and
$\sigma=\sqrt{P''(1)}$. This implies that
\begin{align*} 
    Z=\Ex(\exp(\theta|Y_1|))<\exp(\theta(\max_{-c\leq i \leq
      d}|p_i|\max(c,d)/\sigma))<\infty.
\end{align*}
The conditions of Theorem~\ref{theo:strong} are verified.
Considering the normalization factor $\sigma\sqrt{n}$ of the height of
bridges of length $n$,
we can write Equation (\ref{eq:stringrayleigh}) as
\begin{equation}
\label{eq:normrayleigh}
\lim_{n\rightarrow \infty}\Pr\left(S_n=0,\max _{0\leq k\leq n}\frac{S_k}{\sigma\sqrt{n}}>x\right)=e^{-2x^2}\times\left(1+O\left(\frac{1}{\sqrt{n}}\right)\right)
                                        \qquad x\in]0,+\infty[,
\end{equation} 
while normalisation of Equation (\ref{eq:strBn}) gives
\begin{equation}
\label{eq:normbrown}
\lim_{n\rightarrow \infty}\Pr\left(\max_{k\leq n}\left|\frac{S_k}{\sqrt{n}}-\frac{B_k}{\sqrt{n}}\right|
       \leq C\frac{\log n}{\sqrt{n}} +\frac{x}{\sqrt{n}}\right)\leq
 Ke^{-\lambda x}.
\end{equation}
The error
term of the distance of {\it any} point $S_i$ to the standard Brownian
limit on $[0,1]$
is 
$O(\log n/\sqrt{n})$.
We obtain for the record point $B_{i(h)}$
where $i(h)$ is the first time at which the discrete walk reaches height
$h$ an error term $O(1/\sqrt{n})$. We observe that the distribution of the height of a
standard Brownian on $[0,1]$ is $e^{-2x^2}$. The point $B_{i(h)}$ is {\em tight}
to the Brownian limit.

We propose the following conjecture that should be  refined and
possibly extended.
\begin{conjecture}
\label{conj:tight}
The highest (resp. lowest) points of long enough positive
(resp. negative) arches of discrete bridges are tight to the Brownian
limit.
\end{conjecture}

\section{Periodic case}
A periodic~\cite{BanderierFlajolet2002} walk of period $p$ and characteristic
polynomial $P(u)$ verifies
\begin{equation}
\label{eq:dPu}
\Pi(u)=u^c P(u)=H(u^p),\qquad \text{with $H(u)$ a polynomial}.
\end{equation}


The fundamental period $p$ is the greatest common divisor of the sequence of
powers of $u$ in the polynomial $\Pi(u)$. If $p=1$ the walk is
aperiodic, elsewhere we have $c+d=kp$ with $k\in\bN$.
\begin{example}
\label{ex:nored}
Let
 $P(u)=u^9+u^3+\dfrac{1}{u^{3}}$, which gives
$\Pi(u)=u^{12}+u^6+1$,
with periods $\{2,3,6\}$ and fundamental period $6$, while $H(v)=v^2+v+1$.
\end{example}
\begin{remark}
\label{rem:pversusc}
Let us consider any walk of fundamental period $p$ and larger negative
(resp. positive)
jump $-c$ (resp. $d$).

\ \\[-2ex]\noindent
Such a walk must verify $p \perp c$ and $p\perp d$  ($\operatorname{gcd}(p,cd)=1$).
\ \\[1ex]\noindent
If not, we have $c=ac'$ and $p=ap'$ with $a\geq 2$.
This implies that
\[u^cP(u)=H(u^p)=u^{ac'}P(u)=H(u^{ap'})\Longrightarrow
     P(u)=H(u^{ap'})/u^{ac'}=Q(u^a),\]
 with $Q(y)$ a Laurent polynomial,
and therefore $P(u)$ is not reduced.

The same argument applies to the dual walk $\widetilde{P}(u)=P(1/u)=\dots+p_pu^{-d}$
when considering $u^d\widetilde{P}(u)$.

 Example~\ref{ex:nored} provides
such a non reduced walk.
\end{remark}
\subsection{Singularities of a periodic walk}
\label{sec:persing}
The polynomial $P(u)$ has minimal period $p$, and we can obtain the values of
the $k$th derivatives of $P(u)$ evaluated at $u=\kappa_{\ell}\tau$ for
$\kappa_{\ell}=e^{2i\pi\ell/p}$ by differentiation of 
$\Pi(u)=H(u^p)=u^cP(u)$ with respect to their values at $u=\tau$.
We also have ${\displaystyle \Pi_k(u)=
  u^k\frac{d^k\Pi(u)}{du^k}=H_k(u^p)}$,
 this gives
\begin{align}
\label{eq:PPu}
  \Pi(\kappa_{\ell}\tau)&=\Pi(\tau)=(\kappa_{\ell}\tau)^{c}P(\kappa_{\ell}\tau)=\tau^cP(\tau) \quad\Rightarrow\quad
  P(\kappa_{\ell}\tau)=\frac{P(\tau)}{\kappa_{\ell}^c}\\
\label{eq:Pprimrk}
  \Pi_1(u)&=u\Pi'(u)=\Pi_1(\kappa_{\ell}u)
        \left\{\begin{array}{l}=c\Pi(u)+u^{c+1}P'(u)\\
             =c\Pi(\kappa_{\ell}u)+(\kappa_{\ell}u)^{c+1}P'(\kappa_{\ell}u)
              \end{array}\right.\Rightarrow
  P'(\kappa_{\ell}\tau)=0
  \\
\nonumber u^2\frac{d^2\Pi(u)}{du^2}&\left\{\begin{array}{l}
      =u^{c+2}P''(u)+2cu^{c-1}P'(u)+c(c-1)u^cP(u)\\
      =(\kappa_{\ell}u)^{c+2}P''(\kappa_{\ell}u)+2c(\kappa_{\ell}u)^{c-1}P'(\kappa_{\ell}u)+c(c-1)(\kappa_{\ell}u)^cP(\kappa_{\ell}u) \end{array}\right.\\
\label{eq:P2k}
&\qquad\qquad\qquad\qquad\quad\Rightarrow\quad P''(\kappa_{\ell}\tau)=\frac{P''(\tau)}{\kappa_{\ell}^{c+2}}.
\end{align}
Assuming now that
\[\Pi_k(u) :=u^k\frac{d^k\Pi(u)}{du^k}=\sum_{0\leq j\leq
  k}\alpha_{k,j}u^{c+k-j}\frac{d^jP(u)}{du^j},\quad\text{and } \left.\frac{d^{k}P(u)}{du^k}\right|_{u=\kappa_{\ell}}=\frac{1}{\kappa_{\ell}^{c+k}}\left.\frac{d^{k}P(u)}{du^k}\right|_{u=\tau},\]
by differentiation of $\Pi_k(u)$, we obtain $\Pi_{k+1}(u)$ that verifies
\[u^{k+1}\Pi_{k+1}(u)=\sum_{0\leq j \leq k+1}\alpha_{k+1,j}u^{c+k+1-j}\frac{d^jP(u)}{du^j}.\]
We make one more times use of the periodicity of the walk.
We have $u^{k+1}\Pi^{(k+1)}(u)=H_{k+1}(u^p)$ and therefore
\begin{equation}
\label{eq:Pkappak}
\left\{\begin{array}{l}
            H_{k+1}(u)=0\ \text{if } k+1>d+c\\
            H_{k+1}((\kappa_{\ell}\tau)^p)=H_{k+1}(\tau^p)\\
            P'(\kappa_{\ell})=P'(\tau)=0,\\
            \end{array}\right. \Longrightarrow
   \left.\frac{d^{k+1}P(u)}{du^{k+1}}\right|_{u=\kappa_{\ell}\tau}=\frac{1}{\kappa_{\ell}^{c+k+1}}\left.\frac{d^{k+1}P(u)}{du^{k+1}}\right|_{u=\tau},
\end{equation}
which leads to the lemma
\begin{lemma}
\label{lem:dPk}
With 
$\kappa_{\ell}=e^{2i\pi\ell/p}$ and $P'(\tau)=0$  we have
\[P(\kappa\tau)=\frac{P(\tau)}{\kappa^c}, 
\quad
P'(\kappa\tau)=0,\quad\left.\frac{d^{k}P(u)}{du^{k}}\right|_{u=\kappa_{\ell}\tau}=\frac{1}{\kappa_{\ell}^{c+k}}\left.\frac{d^{k}P(u)}{du^{k}}\right|_{u=\tau}
\ (k\geq 2).\] 
\end{lemma}
Section~\ref{sec:asing} states by domination that in the aperiodic case
there are no singularities in the open disk $|z|<\rho$;
the same proof applies in the
 periodic case.

We want now to check that $P'(u)$ has on the
circle $|u|=\tau$ no other root $\chi_y\tau=\tau e^{2i\pi y}$ than one of the roots $\kappa_{\ell}\tau$
  for $\ell\in [0..p-1]$.

By the triangle inequality, we have $|P(\chi_y)|= P(\tau)$ only if the
arguments $\alpha_j=2\pi jy$  of the monomials $p_j \tau^j e^{2i\pi jy}$ of $P(u)$ are equal, where
$P(u)=p_d u^d +\dots + p_ju^j +\dots+ p_{-c} u^{-c}$ and $-c \leq j \leq d$.
This implies that $P(\chi_y)=e^{2i\pi m\alpha}P(\tau)$ for some $m\in\bN$
and $\alpha\in\bQ$.

We know from Banderier-Flajolet\cite{BanderierFlajolet2002} proof of the domination of
the kernel roots (see the caption of Figure~\ref{fig:ROOTS}) that
$P(\orr\tau)$ 
is decreasing for $\orr<1$ and increasing for $\orr >1$ and therefore $P'(\orr\tau)=0$
only for $\orr=1$.
As a consequence,
\[\left.\frac{d}{d \orr}P(\orr\tau)\right|_{\orr=1}=0\quad
\Longrightarrow \quad P'(\chi_y)=0\quad \text{(since $P(u)$ is
  analytic)}.\]

We obtain $P'(\kappa_{\ell}\tau)=0$ in Equation (\ref{eq:Pprimrk}) with
the lone assumption that $P'(\tau)=0$. Since $P'(\chi_y\tau)=0$, replacing $\tau$ by
$\tau\chi_y$ in this equation gives
$P'(\kappa\chi_y\tau)=0$.

Similarly to the proof of the Daffodil Lemma
(see~\cite{FlajoletSedgewick2009} Lemma IV.1), if $y$ is irrational 
the sequence $(2\pi(y+j/p) \bmod 2\pi)$ is infinite, 
and therefore  the polynomial $P(u)$ has an infinite number of zeroes, which is
 impossible.

Let us assume now that 
 $y=g+\dfrac{x}{p}$ with $g\geq 0$
integer and $x=\gamma/\delta<1\in\bQ$.

We have with regards to the period $p$
\[ \chi_y^cP(\chi_y)=\tau^c e^{2i\pi c(g+x/p)})P(\tau
e^{2i\pi(g+x/p)})=H(\tau^pe^{2i\pi x})=\tau^c e^{2i\pi cx/p}P(\tau
e^{2i\pi x/p}),
\]
and there is a root $\chi_x$ of $P(u)$  corresponding to the case $g=0$, 
which belongs to the arc $z=\tau e^{2i\pi t/p}$ with $0<t<1$.

Let $\chi_{jx}=\tau e^{2i\pi jx \bmod 2\pi}$. 
There exists an integer $k$ 
and $m=kx\leq \delta$ such that $P(\chi_m)=\tau$.
Let $K=\{\chi_{jm}; j=0\,..\,\delta\!-\!1\}$. 
The set
$K\setminus \tau$ 
has an element $\chi_1$ of smallest argument $2\pi/q$ with $q>p$ and
$q=|K|$,
and therefore $K=\{\tau e^{2i\pi j/q}; j=0\,..\,q\!-\!1\}$.

We have
$|\tau^c\chi_1^cP(\chi_1\tau)|=\tau^cP(\tau)$ and therefore 
$\tau^c\chi_1^cP(\chi)=R(\tau^q\chi^q)$ with $R(u)$ a polynomial. This implies
that $q$ is a period of $P(u)$; since $q>p$, it contradicts  
the hypothesis that $p$ is the fundamental and therefore largest
period of $P(u)$.

We obtain the following lemma.
\begin{lemma}
\label{lem:Pchi}
If the polynomial $P(u)$ has fundamental period $p$, the function $|1/P(u)|$ attains its maximum
$1/P(\tau)$ on the circle $|u|=\tau$ at the
points $\kappa_{\ell}\tau$ where $\kappa_{\ell}=e^{2i\pi\ell/p}$ and
only there. These points verify $P'(\kappa_{\ell}\tau)=0$ and $P''(\kappa_{\ell}\tau)\neq 0$; they are saddle-points.
\begin{equation}
\label{eq:Pchi}
\left|\frac{1}{P(u)}\right| <\frac{1}{P(\tau)}, \qquad\text{for }
u=\tau e^{2i\pi t/p} \text{ and } t\not\in \bZ.
\end{equation}
\end{lemma}

\subsection{Unbounded periodic bridges}
\label{sec:unboundper}
Following Banderier-Flajolet~\cite{BanderierFlajolet2002}, Section~\ref{sec:unbaperiodic}
enumerates aperiodic unbounded bridges by the saddle point estimate.
As specified by Lemma~\ref{lem:Pchi} we have $p$ saddle points
  $\kappa_{\ell}\tau=\tau\exp(2i\pi\ell/p)$ on the
circle
$|u|=\tau$ and as previously mentioned, there are no critical points $\zeta=1/P(\upsilon)$
such that $|\upsilon|<\tau$ and $P'(\upsilon)=0$.

Let $\mathcal{S}_{\ell}$ be the saddle-point integral giving the
contribution of $[u^0]P(u)^n$ in a suitable small  neighborhood 
$\mathcal{V}_{\ell}=\{z=\kappa_{\ell}\tau e^{is}, s\in [-\nu..\nu]\}$ of the point $\kappa_{\ell}\tau$.
We refer to Flajolet-Sedgewick 
and Greene-Knuth books~\cite{FlajoletSedgewick2009,GreKnu81} for
detailed proofs.

Using Lemma~\ref{lem:dPk} we obtain with $\mathcal{S}_0=V_n$ of 
Equation (\ref{eq:unboun0})
\begin{equation}
\label{eq:Sell}
\mathcal{S}_{\ell}\sim\frac{1}{2\pi}\int_{-\nu}^{\nu}
P^n(\kappa_{\ell}\tau e^{is})ds
  \sim\frac{1}{\rho^n\kappa_{\ell}^{cn}}\times\left(
\frac{1}{2\sigma\tau\pi\sqrt{n}}+\dots\right) = \kappa_{\ell}^{-cn}\mathcal{S}_0 .
\end{equation}
 
The integers $p$ and $c$ are relative
primes by Remark~\ref{rem:pversusc}. Since $\kappa_{\ell}=\kappa^{\ell}$ with $\kappa=e^{2\pi i/p}$,
we have the set equation
\begin{equation}
\mathcal{C}_{\ell}=\{c\ell \bmod
p,\ \ell\in[0,1,..,p-1]\}=\{\ell,\ \ell\in[0,1,..,p-1]\}.
\end{equation}
If $n \bmod p=b\neq 0$,
 we have also $-cn \bmod p =-cb \bmod p=j'\neq 0$ with $j'<p$. Therefore
\[\kappa_{\ell}^{-cn}=\kappa^{j'\ell}\qquad \text{with } j'<p,
\ell<p.\]
\begin{itemize}
\item
If $j'$ divides $p$ we have $p=aj', \ a<p$ and $j'\ell=\ell\times p/a$.
\[\mathcal{C}'_{\ell}= \{j'\ell \bmod p, \ell\in[0,1,..,p-1]\} = \{\ell', \ell'\in
   [0,p/a,2p/a,..,p-1]\}\quad \text{and}\quad
   |\mathcal{C}'_{\ell}|=a.\]
While $\ell$ goes through the sequence $(0,1,..,p-1)$, the
integer $j'\ell \bmod p$ repeats $a$ times the sequence 
$(0,p/a,2p/a,..,p-1)$ and  the sum of terms
$\kappa^{\ell p/a}$ along this last sequence is $0$.
\item 
else
\[ \{j'\ell \bmod p, \ell\in[0,1,..,p-1]\} = \{\ell', \ell'\in [0,1,..p-1]\}.\]
\end{itemize}
In both cases, if $n \bmod p \neq 0$ we obtain as expected
${\displaystyle \sum_{0\leq\ell<p}\mathcal{S}_{\ell}=0}$; 
when $n$ is not a multiple of $p$ there are no bridges of
length $n=mp+b$ with $0<b<p$.\\[2ex]\noindent
On the contrary, when  $n=mp$, since $\kappa_{\ell}^{-cmp}=1$,  we obtain
${\displaystyle b_n^{<\infty}=\sum_{\ell=0}^{p-1}\mathcal{S}_{\ell}=p\times\mathcal{S}_0,}$
where $\mathcal{S}_0$ is defined as before.

\begin{figure}[t!]
\begin{center}
\begin{tikzpicture}[scale=1.5]
\foreach \z in {0,1,2,3,4} {
\begin{scope}[rotate around={(\z)*-72:(0,0)}]
\path(4,0) coordinate (X2);
\path(0,-1.5) coordinate (Y1);
\path(0,1.5) coordinate (Y2);
\path(-2,0) coordinate (X1);
\path(-1,1.8) coordinate (R);
\path(0,0) coordinate (origin);
\path(1,0) coordinate (P1);
\path(0.7071,0.7071) coordinate (P2);
\path(0.97,0.1456) coordinate (rplus);
\path(0.97,-0.1456) coordinate (rminus);
\path(3.0,0.1456) coordinate (Rplus);
\path(2.0,-1.3) coordinate (Ccontour);
\path(2.0,0.2426) coordinate (Rplusmean);
\path(3.0,-0.1456) coordinate (Rminus);
\path(2.0,-0.1456) coordinate (Rminusmean);
\path(3.5,-0.1) coordinate (tauminus);
\path(3.5,0.1) coordinate (tauplus);

\path(-0.15,-0.30) coordinate (OO);
\path(-1.3,-1.3) coordinate(GAMMA);
\path(2,2.2) coordinate (rtozero);
\path(-1.3,2) coordinate (gamma);

\draw[->,green,thick] (rplus) -- (Rplus);
\draw[->,green,thick] (Rminus) -- (rminus);
\draw[->,red,thick] (Rplus) -- (Rminus);
(Rplusmean);
\draw[-] (tauminus) -- (tauplus);

\fill[red] (origin) circle(0.05);
\fill[blue] (rplus) circle(0.05);
\fill[blue] (rminus) circle(0.05);

\fill[blue] (Rplus) circle(0.05);
\fill[blue] (Rminus) circle(0.05);
\node at (1.2,0.40) {$R^+_{\z}$};
\node at (1.1,-0.40) {$R^-_{\z}$};
\node at (3.0,0.40) {$S^+_{\z}$};
\node at (3.0,-0.40) {$S^-_{\z}$};
\draw[->] (origin) -- (X2);
\path(3.5,0.18) coordinate (tautau);
\node at (tautau) {$\rho_{\z}$};
\end{scope};
};
\node at (OO) { $O$};
\node at (4.0,0.15) {$x$};
\node at (-2.0,3)
      {$\widehat{\mathcal{C}_r}=\overrightarrow{\gamma^+_0\gamma^{\perp}_0
\gamma^-_0\Gamma_{r,0}\dots\gamma^-_4\Gamma_{r,4}}$};
\node at (3.0,3.1) {$\kappa_{\ell}=e^{2i\pi\ell/5}$};
\node at (3.0,2.5) {$\rho_{\ell}=\kappa_{\ell}\rho_0$};
\node at (3.5,1.35)
      {$|R^+_{\ell}R^-_{\ell}|=|S^+_{\ell}S^-_{\ell}|=2s,\ \ s=o(r^{3n})$};
\node at (3.3,1.9) {$|OR^+_{\ell}|=|OR^-_{\ell}|=r$};
\node at (4.0,-3.0)
      {$\Gamma_{r,\ell}=\overrightarrow{R^-_{\ell'}R^+_{\ell'+1}}\quad (\ell'=p-\ell
     \bmod p)$};
\node at (3.0,-1.5) {$\gamma^+_{\ell}=\overrightarrow{R^+_{\ell}S^+_{\ell}}$};
\node at (3.0,-2) {$\gamma^-_{\ell}=\overrightarrow{S^-_{\ell}R^-_{\ell}}$};
\node at (3.0,-2.5) {$\gamma^{\perp}_{\ell}=\overrightarrow{S^+_{\ell}S^-_{\ell}}$};
\draw[->] (origin) -- node[right=1pt,below=0.2pt] {$r$} (P2);
\foreach \y in {0,1,2,3,4} \draw
         [red,thick,domain=(\y*72)+8.54:(\y+1)*72-8.54]
         plot ({cos(\x)},{sin(\x)});
\node at (-2.0,-3.0) {$\gamma_{\ell}=\lim_{r\rightarrow 0}
          \overrightarrow{(\kappa_{\ell}r)( \kappa_{\ell}\rho_0})$};
\end{tikzpicture}
\end{center}
\begin{picture}(0,0)(0,0)
\put(0,-40){\rule{\textwidth}{0.2mm}}
\end{picture}
\caption{\label{fig:duchoncontour}Integration contour for the Duchon
  walk $P(u)=u^2+\dfrac{1}{u^3}$ with period $5$.}
\end{figure}

\subsection{Preliminary Cauchy contour for the periodic case}
\label{sec:percontou}
In the periodic case, the preliminary Cauchy contour is star-shape 
and later deformed by the usual method of singularity analysis  to $p$
dominant Hankel contours and negligible secondary ones ; with
$\kappa_{\ell}=e^{2i\ell\pi/p}$, the $\ell$th Hankel 
contour $\gamma_{\ell}$ comes from $\kappa_{\ell}(+\infty)$
winds around the point $\kappa_{\ell}$ and goes back to
$\kappa_{\ell}(+\infty)$.
We will prove that the Hankel integral along the path $\gamma_{\ell}$
is equal to the one along $\gamma_0$.
This will induce a multiplicative factor $p$ occurring in $b_n^{>x\sigma\sqrt{n}}$ and in $b_n^{<\infty}$;
this factor cancels when taking the ratio of the two quantities.

We make $p-1$ successive rotations of angle $2\pi i/p$ of the path
$\cP=\overrightarrow{R^+_0S^+_0S^-_0R^-_0}$,  which generates the paths
$(\cP_1,\dots\cP_{\ell},\dots\cP_{p-1})$, where
\[\cP_{\ell}=\overrightarrow{R_{\ell}^+S_{\ell}^+S_{\ell}^-R_{\ell}^-},\qquad 
      \left\{\begin{array}{l}
        R_{\ell}^+=e^{2\pi i \ell/p}R_0^+,\quad S_{\ell}^+=e^{2\pi i \ell/p}S_0^+,\\
        R_{\ell}^-=e^{2\pi i \ell/p}R_0^-,\quad S_{\ell}^-=e^{2\pi i \ell/p}S_0^-
      \end{array}\right.
\]
and the new contour $\widehat{\cC}_{r}$ is completed by the
$p$ arcs $\Gamma_{r,\ell}=\overrightarrow{R_{\ell}^- R_{\ell+1\!\!\mod p}^+}$ of radius $r$, 
where we note
\[R_0^-=R^-,\ R_0^+=R^+,\ S_0^-=S^-,\ S_0^+=S^+,  \qquad \text{with
}R^+,\ R^-, \ S^+, \ S^- \text{ defined as in Figure~\ref{fig:contour}}.\]
See Figure~\ref{fig:duchoncontour} for the case $p=5$.

\subsection{Dominant singularities properties for the periodic case}
\label{sec:perdom}
With a walk of period $p$ and $x\in ]0,\rho]$, we have as in the
non-periodic case~\footnote{See Figure~\ref{fig:ROOTS}.} 
a number
$\tau$ such that $P(u)$ is decreasing for $x<\tau$ and increasing for
$x>\tau$; this number $\tau$ verifies as in the aperiodic case
$P'(\tau)=0$; let  $\rho_0=\rho=1/P(\tau)$.

Since $\kappa_{\ell}=e^{2i\pi\ell/p}$ and $u^c P(u)=H(u^p)$, we have
for $v \in \bR^+$,
\begin{equation}
\frac{\kappa_{\ell}^c v^c}{z}=\kappa_{\ell}^c v^c
P(\kappa_{\ell}v)=H(\kappa_{\ell}^p v^p)=H(v^p)=v^cP(v),\quad 
\text{ and }
P(\kappa_{\ell}v)=\kappa_{\ell}^{-c}P(v)
\end{equation}
Therefore $\kappa_{\ell}^cP(\kappa_{\ell} v)$ is real for $v\in \bR^+$
and so  is
 the real equation 
\begin{equation}
\label{eq:dreal}
 Z=\kappa_{\ell}^{-c}z=\frac{1}{P(v)},\qquad z=x\kappa_{\ell}^c,
\quad x\in]0,\rho[
\end{equation}
which has as $Z\rightarrow 0^+$ small and large roots $u_{i,\ell}(Z)$ and
$v_{j,\ell}(Z)$. 
We prove next that they verify
 the same properties as the small $u_i(z)$ and large roots 
$v_j(z)$ for
$z\in]0,\rho[$
in the aperiodic case.

The triangle inequality of Equation (\ref{eq:triangle}),
\[ |P(re^{it})|<P(r) \quad \text{ for } 0<r<\rho,\ t\not\equiv 0 \pmod{2\pi}\]
 is
no more verified since $P(\kappa^{\ell}w)=P(w)$ for
$\kappa=e^{2i\pi/p}$, with $\ell$ an integer and $w$ any solution of
$1-zP(w)=0$.

Let $\mathcal{K}_{p,\ell}$ be the cone
\[\mathcal{K}_{p,\ell}= \left\{z=xe^{it},\quad x\in]0,\rho[,
\quad t\in \left[2\pi\frac{\ell}{p},2\pi\frac{\ell+1}{p}\right[\ \ \right\} \]
Within the cone $\mathcal{K}_{p,\ell}$, the triangle identity is valid,
\begin{equation}
\label{eq:pdom}
|P(xe^{it})|<P(x) \quad \text{for }
xe^{it}\in\mathcal{K}_{p,\ell}.
\end{equation}
Within these restricted domains, the proofs of Lemma 1 and 2  of
Banderier-Flajolet~\cite{BanderierFlajolet2002} of aperiodic domination
(Lemma~\ref{lem:domin}) apply 
to the roots of the real
equation (\ref{eq:dreal}), which leads to the lemma.
\begin{lemma}
\label{lem:pdomin}
In the periodic case, with $u_{i,\ell}(Z)$ (resp. $v_{j,\ell}(Z)$)
 the small (resp. large) roots
for $z=x\kappa_{\ell}^c$, with $Z=\kappa_{\ell}^{-c}z$ and $x\in]0,\rho]$,
 along each segment $O\rho_{\ell}$ the roots of the
kernel equation $1-zP(u)=0$ verify
\begin{equation}
|u_{i,\ell}(Z)|< u_{1,\ell}(Z)<v_{1,\ell}(Z)< |v_{j,\ell}(Z)|,\qquad
(i\neq 1,\ j\neq 1),\qquad Z\in ]0,\rho[.
\end{equation}
They verify the same analytic properties as the roots in the
aperiodic case.

The dominant  small (resp. large) root $u_{1,\ell}(Z)$ (resp. $v_{1,\ell}(Z)$) is an  analytic solution which
can be continued from the dominant real  small (resp. large) root at  $0$ in the
direction $O\kappa_{\ell}$. 
\end{lemma}
Lemma~\ref{lem:dPk}
leads to an asymptotic expansion of $1/P(u)$ in the
neighborhood of $u=\kappa_{\ell}\tau$, 
and we have $\left.\dfrac{d^kP(u)}{du^k}\right|_{u=\kappa_{\ell}\tau}=\dfrac{1}{\kappa_{\ell}^{c+k}}\left.\dfrac{d^{k}P(u)}{du^{k}}\right|_{u=\tau}$ ,
\begin{equation}
\label{eq:onekappa}
z=\kappa_{\ell}^{c}Z=\frac{\kappa_{\ell}^c}{P(u)}=\frac{1}{P(\tau)}-\frac{1}{2}\frac{P''(\tau)}{P^2(\tau)}\frac{(u-\kappa_{\ell}\tau)^2}{\kappa_{\ell}^2}
   +\sum_{j\geq
     3}\alpha_j\left.\dfrac{d^jP(u)}{du^j}\right|_{u=\tau}\frac{(u-\kappa_{\ell}\tau)^j}{\kappa_{\ell}^j},
\end{equation}
where the coefficients $\alpha_j$ are functions of the derivatives of
$P(u)$ evaluated at $u=\kappa_{\ell}\tau$.

The first terms of the preceding expansion give with
 $u=\kappa_{\ell}U$, $P''(\tau)=\sigma^2$ and $\rho=1/P(\tau)$,
\begin{equation}
U_1(Z)=\tau-\frac{\sqrt{2(1-Z/\rho)}}{\sqrt{\rho\sigma}}+O(1-Z/\rho),\quad 
        V_1(Z)=\tau+\frac{\sqrt{2(1-Z/\rho)}}{\sqrt{\rho\sigma}}+O(1-Z/\rho),\quad
        Z/\rho\sim 1^-.
\end{equation}
From there we recover the expression of the dominant small and large
roots on the path $\gamma_{\ell}=\overline{O\kappa_{\ell}^c}$.
\begin{equation}
\label{eq:U1V1}
\left.\begin{array}{l}
 u_{1,\ell}={\displaystyle\kappa_{\ell}\times U_1(\kappa_{\ell}^{-c}z)=
     \kappa_{\ell}\left(\tau-\frac{\sqrt{2(1-z/\kappa_{\ell}^c\rho)}}{\sqrt{\rho\sigma}}\right)+O\left(1-\frac{z/\kappa_{\ell}^c}{\rho}\right)}\\[3ex]
 v_{1,\ell}={\displaystyle\kappa_{\ell}\times V_1(\kappa_{\ell}^{-c}z)=
     \kappa_{\ell}\left(\tau+\frac{\sqrt{2(1-z/\kappa_{\ell}^c\rho)}}{\sqrt{\rho\sigma}}\right)+O(\left(1-\frac{z/\kappa_{\ell}^c}{\rho}\right)}
\end{array}\right|\quad Z=z/\rho\kappa_{\ell}^c \sim 1^-
\end{equation}

Equations (\ref{eq:QoverQ}) and (\ref{eq:uvh}) become
\begin{equation}
\label{eq:QovQz}
\frac{Q(u_{1,\ell}(z))}{Q(v_{1,\ell}(z))}=1+O\left(\sqrt{1-z\kappa_{\ell}^c/\rho}\right),
 \qquad Z=z/\kappa_{\ell}^c \sim \rho^-,
\end{equation}

\subsection{Integrations  along the paths $\Gamma_{r,\ell}$}
\label{sec:extzero}
As stated previously if a walk of characteristic polynomial $P(u)=a_d u^d+\dots+a_{-c}u^c$ has
period $p$,
 the natural integer $p$ is the largest
common divisor of the set of powers of $u$ in $u^cP(u)$; by the Bezout
theorem, any linear combination $L$ with positive integer coefficients of
the integers $d,d-1,\dots,-c+2,-c$ verifies $L= 0\bmod p$ if and only
if $p=\operatorname{pgcd}(d,d-1,\dots,-c)$.

The function $B_h(z)=z\overline{B}_h(z)=[u^0]F^{[>h]}(z,u)$ of Equation
(\ref{eq:overB})
 counts the number of
bridges of height above $h$. The function
$[u^0](1-zP(u))^{-1}=\sum_{i\geq 1}q^i z^i$  counts all the bridges and by the preceding remark
its non null coefficients $q_i$ verify $i=mp$ for $m\in\bN$; since
this function dominates term by term $B_h(z)$, we have
$B_h(z)=\widehat{B_h}(z^p)$ 
with $\widehat{B_h}(z)$ analytic at $0$ (see
\cite{BanderierFlajolet2002}, section 3.3).

Let us consider a walk of length $n$ with $n$ large.

The sequence of jumps 
$\overbrace{+d+d\dots+d}^{\lceil
  2x\sqrt{n}\rceil}\overbrace{-c-c\dots -c}^{\lfloor 3x d/c\times \sqrt{n} \rfloor}$
 reaches height $2xd\sqrt{n}$ and terminates at a
negative ordinate. This implies that there is at least a walk that
reaches the $x$-axis at time $t\leq \lceil(2x+3xd/c)\sqrt{n}\rceil$.

\subsubsection{Height $h$ as an  integer}
\label{sec:hinteger}
The  preceding paragraph entails that the expansion of $B_h(z)$ at $z=0$ is therefore if
$h$ is an integer
\begin{equation}
B_h(z)=z\sum_{k=1}^d\sum_{j=1}^c \left(\frac{u_j}{v_k}\right)^h \frac{Q_k(u_j)}{Q_k(v_k)}\frac{u'_j}{v_k}=b_m z^{mp}+O\left(z^{(m+1)p}\right),\qquad \text{with } mp\leq t,
\end{equation}
which gives
\begin{align*}
\mathcal{I}_r&=\frac{1}{2i\pi}\int_{\cup_{\ell}
  \Gamma_{r,\ell}}\frac{B_h(z)}{z^{pn+1}}dz=b_m\sum_{0\leq \ell\leq
  p-1}\mathcal{I}_{r,\ell}\\
\text{where}&\quad
\mathcal{I}_{r,\ell}=\frac{1}{2i\pi}\int_{\Gamma_{r,\ell}}z^{-(pn+1)}z^{pm}(1+O(z^{(m+1)p})dz
\end{align*}
and $b_m$ is upper bounded by $V_{pm}$, the number of  unbounded
bridges of length $pm$ (see Theorem~\ref{theo:BaFlaaperunb}),
which implies $b_m=O(P(\tau)^{pm})$. 

The change of variable
$z=re^{i\nu}$ leads as $r\rightarrow 0$ to 
\begin{align*}
\label{eq:zpm}
b_m\int_{\Gamma_{r,\ell}}\frac{ z^{pm}}{z^{pn+1}}dz
&=b_m\int_{2\pi\ell/p+s}^{2\pi(\ell+1)/p-s}r^{-(pn-pm)}e^{(pm-pn)i\nu}d\nu\\
&=b_m\frac{r^{-(pn-pm)}}{p(m-n)i}\left[e^{p(m-n)i\nu}\right]_{2\pi\ell/p+s}^{2\pi(\ell+1)/p-s}\\[2ex]
&=r^{-(pn-pm)}\left(2b_m s+O(b_m p^2(m-n)^2s^3)\right)=O(r^{2pn}) \quad \text{for } s=o(r^{3pn}). 
\end{align*}
The error term follows immediately.
\subsubsection{Non-integer $h=x\sigma\sqrt{n}$}
\label{seq:hpernoint}
If we consider a non-integer  $h=x\sigma\sqrt{n}$ we can turn to the
method used in the aperiodic case by  following  the approach\footnote{See
 also \cite{FlajoletSedgewick2009} Section VII.7.1, and the use of a
 local uniformizing  parameter.} of
Banderier-Flajolet~\cite{BanderierFlajolet2002} (Example $5$) which
handles the case of a generalized Duchon walk $P_{d,c}(u)=u^d+u^{-c}$
of period $p=c+d$ with
kernel equation $u^c=z(1+u^{c+d})$. One obtains
\[
u_1(z)=z^{1/c}W_1(z^{p/c})=z^{1/c}\times(1+\alpha_1z^{p/c}+\dots),\quad
v_1(z)=\frac{1}{z^{1/d}}W_2(z^{p/d})=\frac{1}{z^{1/d}}\times(1+\beta_1z^{p/d}+\dots),
\]
where $W_1(Z_1)$ (resp. $W_2(Z_2)$) are series in the variable
$Z_1=e^{p/c}$ (resp. $Z_2=e^{p/d}$). The
expansions of the other roots follow by substitutions
\begin{align*}
u_j(z)&=\omega_j z^{1/c}W_1(\omega_j^{p/c}z^{p/c})=\omega_j
z^{1/c}\times(1+\alpha_1\omega_j^{p/c}z^{p/c}+\dots),\\
v_k(z)&=\frac{1}{\varpi_k^{1/d} z^{1/d}}W_2(\varpi_k^{p/d}z^{p/d})
=\frac{1}{\varpi_k^{1/d}z^{1/d}}\times(1+\beta_2\varpi_k^{p/d}z^{p/d}+\dots).
\end{align*}
In particular $\dfrac{u_j(z)}{v_k(z)}
       =\dfrac{\omega_j}{\varpi_{k}}z^{1/c+1/d}+O\left(z^{p(1/c+1/d)}\right);$
we have $\dfrac{1}{c}+\dfrac{1}{d}=\dfrac{c+d}{cd}=\dfrac{ep}{cd}$ where 
$e>1$ and $p$ divides $c+d$ by 
Remark~\ref{rem:pversusc}  while $p$ is prime with $cd$.
We omit the end of the proof that follows the same steps as in the aperiodic case
 
We get to the following lemma.
\begin{lemma}
\label{lem:gammaper}
As $r\rightarrow 0$ and $s=o(r^{3n})$
\begin{equation}
 {\it(i)} \quad\int_{\cup_{\ell}\Gamma_{r,\ell}}\frac{B(z)}{z^{pn+1}}dz =
    o(r^{n}),\qquad
 {\it(ii)}\quad 
\int_{\cup_{\ell}\gamma^{\perp}_{\ell}}\frac{B(z)}{z^{pn+1}}dz=o(r^{2n}),
\end{equation}
and
\begin{equation}
(iii) \frac{1}{2i\pi}\int_{\widehat{\mathcal{C}}_r}\frac{B(pz)}{z^{pn+1}}dz
     =\sum_{0<\ell<p-1} \frac{1}{2i\pi}
\int_{\gamma_{\ell}^+\cup\gamma_{\ell}^-}\frac{B(z^p)}{z^{pn+1}}dz+o(r^n).
\end{equation}
\end{lemma}
Case ({\it ii}) of this lemma follows as in the aperiodic case from
regularity and continuity arguments. 

Collecting the preceding equations leads us to the following lemma.

\begin{lemma}
\begin{equation}
\label{eq:collect}
\frac{1}{2i\pi}\oint_{\widehat{\mathcal{C}}_{r}}\frac{\overline{B}(z)}{z^n}dz
  =
  \sum_{\ell=0}^{p-1}\frac{1}{2i\pi}\int_{\gamma_{\ell}^+}\!\!\!+\!\!\int_{\gamma_{\ell}^-}\frac{\overline{B}(z)}{z^n}dz+o(r^n).
\end{equation}
\end{lemma}

\subsection{Hankel integration along the path
  $\mathcal{K}_{\ell}:=\overline{\kappa_{\ell}^c\rho,(\kappa_{\ell}^c\infty})$}
\label{sec:inthankelellc}
As $r$ tends to zero within the contour $\mathcal{C}_r$ defined
 in Section~\ref{sec:percontou} the path
 $R_{\ell}^+S_{\ell}^+S_{\ell}^-R_{\ell}^-$ has for limit the segment
 $\gamma_{\ell}=O_{\ell}\rho_{\ell}$
where $O_{\ell}=\kappa_{\ell}r$ (see  Figure~\ref{fig:duchoncontour}).

We want to integrate along the Hankel contour $\widehat{\gamma_{\ell}}$,
which goes ``by below'' from $\kappa_{\ell}\infty^+$ to
$\kappa_{\ell}\rho_0$ winds clockwise around this point and goes back
to $\kappa_{\ell}\infty^+$ ``by above''. 
\begin{equation}
\label{eq:Ilz}
\mathcal{I}_{\ell}=\frac{1}{2\pi
  i}\int_{\widehat{\gamma_{\ell}}}\frac{z}{z^{n+1}}\left(\frac{u_{1,\ell}(z)}{v_{1,\ell}(z)}\right)^h\frac{Q_1(u_{1,\ell}(z))}{Q_1(v_{1,\ell}(z))}\frac{u'_{1,\ell}(z)}{v_{1,\ell}(z)}dz+O(A^h)\quad\text{
  with } z=\kappa_{\ell}^c Z, \ Z\in\bR^+.
\end{equation}
The change  of variable $z
=Z(t)=\kappa_{\ell}^c\rho\left(1-\dfrac{t}{n}\right)$ gives as
$n\rightarrow \infty$
\[\frac{u'_{1,\ell}(Z(t))}{v_{1,\ell}(Z(t))}Z'(t)dt=
-\frac{1}{\tau\sqrt{2\rho\sigma}\sqrt{t}\sqrt{n}}
\times\left(1+O\left(\frac{1}{\sqrt{n}}\right)\right),\]
Integration of $\mathcal{I}_{\ell}$ (along the 
contour $\widehat{\gamma}_{\ell}$) follows {\it
  verbatim} the lines of integration of 
Section~\ref{sec:semilarge}; therefore we have
\begin{equation}
\label{eq:finIell}
\mathcal{I}_{\ell}=\kappa_{\ell}^{-cn}\mathcal{I}_0
\end{equation}
The integers $p$ and $c$ are relative
primes by Remark~\ref{rem:pversusc}. Since $\kappa_{\ell}=\kappa^{\ell}$ with $\kappa=e^{2\pi i/p}$,
we have again the set equation
\begin{equation}
\mathcal{S}_{\ell}=\{c\ell \bmod
p,\ \ell\in[0..p-1]\}=\{\ell,\ \ell\in[0..p-1]\}.
\end{equation}
Therefore the contour $\widehat{\mathcal{C}_r}$ defined in
Section~\ref{sec:percontou} is completely scanned through as
$\ell$ goes along the integers $0$ to $p-1$.

The discussion terminating Section~\ref{sec:unboundper} applies
identically, which gives
\begin{equation}
\label{eq:finper}
b_{mp}^{>h}=\sum_{\ell=0}^{p-1}\mathcal{I}_{\ell}=p\mathcal{I}_0+o(r^n), \qquad
\lim_{r\rightarrow 0,s=o(r^{3n})}b_{mp+b}^{>h}=0\ (b<p),
\end{equation}
where $\mathcal{I}_0/b_n^{<\infty}$ is given as in
Section~\ref{sec:semilarge}, and we get at first order
\begin{equation}
\label{eq:perdfdineq}
\mathcal{I}_0=\frac{\rho^n}{\sigma\sqrt{2\pi n}}e^{-2x^2\rho/\tau^2}
 \times\left(1+O\left(\frac{1}{\sqrt{n}}\right)\right).
\end{equation}
We conclude the periodic case by the theorem.
\begin{theorem}
\label{theo:rayleighperiodic}
For a set of jumps of period $p$, if $n=mp\rightarrow\infty$,  the probability $\beta_n^{>x\sigma\sqrt{n}}$ that a bridge of length $n$
goes upon the barrier $y=x\sigma\sqrt{n}$ follows a Rayleigh limit
law for $x\in ]0,+\infty[$
\begin{equation}
\beta_n^{>x\sigma\sqrt{n}}=\frac{b_n^{>h}}{b_n^{<\infty}}= 
  \frac{p\mathcal{I}_0}{p\mathcal{S}_0}\times(1+O(B^n))=
 e^{-2x^2\rho/\tau^2} 
 \times\left(1 +O\left(\frac{1}{\sqrt{n}}\right)\right),\quad (B<1).
\end{equation}
\end{theorem}

\section{{\L}ukasiewicz bridges}
\label{sec:lukabridge}
When considering the case of {\L}ukasiewicz walks, where the only
negative jump is $-1$, the characteristic
polynomial verifies
\[P(u)=p_d u^d+\dots+\frac{p_{-1}}{u},\]
and we can obtain as in Banderier-Nicodeme~\cite{BaNi2010} more precise asymptotics for
the convergence to the Rayleigh law.

By differentiation of $K(z,u)=1-zP(u(z))=0$ with
respect to the variable $z$, we obtain that for any solution $u(z)$ of
$K(u,z)=0$
\[ \frac{\partial}{\partial z} (1-zP(u(z))=-P(u(z))-z\frac{\partial
  P(u)}{\partial u}u'(z) \Longrightarrow \frac{\partial P(u)}{\partial
  u}=-\frac{1}{z^2u'(z)}.\]
We also have since $\dfrac{1}{p_d z}u K(z,u)$ is a monic polynomial 
\[ Q_1(u)=\prod_{2\leq i \leq d}(u-v_i(z))= \frac{u(1-zP(u))}{p_dz(u-u_1(z))(u-v_1(z))},\]
and, therefore:
\begin{equation}
\label{eq:Qu1}
{Q_1}(u_1(z))=\left. \frac{1}{p_d z} \frac{\partial}{\partial u}
 \frac{u(1-zP(u))}{u-v_1(z)}\right|_{u=u_1(z)}
 = \frac{1}{p_dz^2}\frac{ u_1(z)}{u'_1(z)(u_1(z)-v_1(z))}\,.
\end{equation}
The value of ${Q}_1(v_1)$ follows by interchanging the r\^oles of $u_1$
and $v_1$.
The integral equation (\ref{eq:Bhsimp}) thus becomes (in
the aperiodic case)
\begin{equation}
\label{eq:u0Fgeqh2}
b_{n,luka}^{[>h]}=\frac{1}{2\pi i}\oint \frac{1}{z^{n+1}}z \left(\frac{u_1(z)}{v_1(z)}\right)^h
 \times\frac{-v'_1(z)u_1(z)}{v_1(z)^2}dz\times (1+O(A^h))\qquad (A<1).
\end{equation}
This equation leads to more precise expansions of the probability
$b^{>h}_n$ that we consider in Section~\ref{sec:luka}.

The periodic {\L}ukasiewicz case is handled in a similar way to the
general periodic case.

We have now
\begin{equation}
\label{eq:perluka}
\mathcal{I}_{\ell,luka}=\frac{1}{2\pi i}\int_{\widehat{\gamma_{\ell}}}\frac{z}{z^{n+1}}\left(\frac{u_{1,\ell}(z)}{v_{1,\ell}(z)}\right)^h\times\frac{-v'_{1,\ell}(z)u'_{1,\ell}(z)}{v^2_{1,\ell}(z)}dz+O(A^h)\quad\text{ with } z=\kappa_{\ell}^c Z, \ Z\in\bR^+.
\end{equation}
Following the same steps of proof as in
Section~\ref{sec:inthankelellc}, we obtain for a walk of period $p$,
\begin{equation}
\label{eq:finperlk}
\lim_{r\rightarrow 0}b_{mp,luka}^{>h}=\sum_{\ell=0}^{p-1}\mathcal{I}_{\ell,luka}=p\mathcal{I}_0, \qquad
\lim_{r\rightarrow 0}b_{mp+b}^{>h}=0\ (b<p),
\end{equation}
where $\mathcal{I}_0/b_n^{<\infty}$ is given as in Section~\ref{sec:semilarge},
\begin{equation}
\label{eq:perdfdineqlk}
\mathcal{I}_0=\frac{\rho^n}{\sigma\sqrt{2\pi n}}e^{-2x^2\rho/\tau^2}
 \times\left(1+O\left(\frac{1}{\sqrt{n}}\right)\right).
\end{equation}
\subsection{Occurrences of Hermite polynomials}
\label{sec:hermitte}
We mention here the occurrences of Hermite polynomials in the asymptotic
expansion of the tail distribution of the height of {\L}ukasiewicz
bridges.

In Equation (\ref{eq:b_nbeforehankel}) and in
the subsequent equations the
speed of convergence factor
 ${\displaystyle\left(1+O\left(\frac{1}{\sqrt{n}}\right)\right)}$ refers only
  to the variable $b_n^{>x\sigma\sqrt{n}}$.
The same lines of proof leads to
\begin{equation}
\label{eq:expx2}
e^{ -2 x^2}=\frac{1}{2i\sqrt{ \pi}}\oint_{\mathcal{H}_1} 
 \frac{1}{\sqrt{2}}\frac{e^t
   e^{-2x\sqrt{2t}}}{\sqrt{t}}dt
\end{equation}
Differentiating with respect to the variable $x$ the right member of Equation
(\ref{eq:expx2}) 
is permitted since the integrand is absolutely converging.
Differentiating  repetitively both sides of this equation
 with respect to the variate $x$ induces
derivatives of $e^{-2x^2}$ and integrals of the type
\begin{equation}
\label{eq:Ir}
I_{r}=\frac{1}{2i\sqrt{\pi}}\oint\frac{1}{\sqrt{2}} e^t
  e^{-2x\sqrt{2t}}t^{r/2}dt,\qquad r\in\{-1\} \cup \bN.
\end{equation}
These expressions can be computed
 by expansions of the exponential functions similar to those done
 for $b_n^{>x\sigma\sqrt{n}}$ in Equations (\ref{eq:I1afterhankel},\ref{eq:xrhotau}).

We have
\begin{equation}
\frac{1}{2i\sqrt{\pi}\sqrt{2}}\frac{d^{\rule{1pt}{0pt}r}}{dx^{r}}\oint\frac{ e^t
  e^{-2x\sqrt{2t}}}{\sqrt{t}}dt=(-2\sqrt{2})^r I_{r-1}\quad 
\text{and}\quad\frac{d^{\rule{1pt}{0pt}r}}{dx^{r}}e^{-2 x^2} = Q_r(x)\times e^{-2x^2}=I_{r-1},
\end{equation}
and therefore
${\displaystyle
I_{r}=\frac{(-1)^{r+1}}{(2\sqrt{2})^{r+1}}Q_{r+1}(x)e^{-2x^2},
}$
where $Q_r(x)$ is a polynomial which can be computed by the
recurrence
$\quad Q_{r+1}(x)=-4x Q_{r}(x)+Q'_{r}(x)\quad \text{with}\quad  Q_0(x)=1.$

The first values of $Q_r(x)$ verify:\\
\small
\begin{center}
\begin{tabular}{|r|c|}\hline
$r$ & $Q_r(x)$\\\hline
$                              0$&$ 1$\\
$                            1$&$ -4 x$\\
$                       2$&$ 16 x^2- 4$\\
$                     3$&$ -64 x^3  + 48 x$\\
$                   4$&$ 256 x^4  - 384 x^2  + 48$\\
$                 5$&$ -1024 x^5  + 2560 x^3  - 960 x$\\
$             6$&$ 4096 x^6  - 15360 x^4  + 11520 x^2  - 960$\\
$           7$&$ -16384 x^7  + 86016 x^5  - 107520 x^3  + 26880 x$\\
$        8$&$ 65536 x^8  - 458752 x^6  + 860160 x^4  - 430080 x^2  +
        26880$\\\hline
\end{tabular}
\end{center}
\normalsize
We observe that $Q_i(x)=(-1)^i\operatorname{He}_i(4x)$, where $\operatorname{He}_i(x)$ is the
probabilist's Hermite
polynomial of index $i$, with recurrence
\begin{equation}
\label{eq:He}
\operatorname{He}_{i+1}(x)=x\operatorname{He}_i(x)-\operatorname{He}_i'(x).
\end{equation}
Therefore we have for $I_r$ of Equation (\ref{eq:Ir})
\begin{equation}
\label{eq:solIr}
\frac{I_r}{\sqrt{\pi}}=\frac{1}{2\pi i}\oint\frac{1}{\sqrt{2}} e^t
  e^{-2x\sqrt{2t}}t^{r/2}dt
     =\frac{1}{\sqrt{\pi}}e^{-2x^2}\frac{1}{(2\sqrt{2})^{r+1}}\operatorname{He}_{r+1}(4x),
\quad r\in \{-1 \} \cup \bN.
\end{equation}
We have the equivalent mappings for the last equation
\begin{equation}
\label{eq:mapts}
             \left\{ t^{r/2}\leadsto \operatorname{He}_{r+1} \text{ with } r\geq -1\right\} 
      \qquad\equiv \qquad
             \left\{ s^r   \leadsto \operatorname{He}_{r+1} \text{ with } s\geq -1\right\};
\end{equation} 
the latter mapping (corresponding to $t=s^2$) is easier to manage with
Maple, while unable to use with the Hankel transform.

\subsection{Detailed asymptotic for {\L}ukasiewicz walks}
\label{sec:luka}
We assume in this section that $P(1)=1$ and $P'(1)=0$, and therefore
$\tau=\rho=1$.

Writing in a neighborhood of $z=1$  the algebraically
 conjugate roots $u_1(z)$ and
$v_1(z)$ as
\begin{align}
\label{eq:u1v1z}
u_1(z)&=1-\frac{\sqrt{2}\sqrt{1-z}}{\sigma}+\sum_{i\geq
  2}a_i\left(\sqrt{1-z}\right)^i,\\
v_1(z)&=1+\frac{\sqrt{2}\sqrt{1-z}}{\sigma}+\sum_{i\geq 2}(-1)^i
a_i\left(\sqrt{1-z}\right)^i,\\
\label{eq:uuovvv}
\frac{u_1(z)}{v_1(z)}&=1-2\frac{\sqrt{2}\sqrt{1-z}}{\sigma}+\sum_{i\geq 2}b_i\left(\sqrt{1-z}\right)^i,
\end{align}
we can compute the coefficients $a_i$
\begin{enumerate}
\item either by plugging a bounded expansion of $u_1(z)$ in
  $K(z)=1-zP(u_1(z))=0$, taking an expansion of $K(z)$ at $z=1$,
and identifying iteratively the coefficients (as expected, they are
functions of the derivatives of $P(u)$ evaluated at $u=\tau$);
\item more efficiently, by using
 Newton~\footnote{by instance, the algebraic inverse of $z=P(1-v)$ is obtained by 
using the change of variable $z=1-t/n=1-X^2$ and by initializing the
iteration with $+\sqrt{2}X/\sigma$, (resp. $-\sqrt{2}X/\sigma$) which gives
$u_1(1-X^2)$, (resp. $v_1(1-X^2))$.}
  iteration~\cite{aecf-2017-livre};
\item even faster, by using the
  function {\tt gfun:algeqtoseries} of the package~\footnote{Available
  at Bruno
  Salvy's website.} {\tt gfun}~\cite{SalZim1994}
   if the coefficients of $P(u)$ are given as numeric rational; this function
 computes a series  expansion at the origin of a solution of an
  algebraic equation. 
\end{enumerate}
In particular, the expansions of the variables
$\widetilde{u}_1(t)=u_1(1-t/n)$ and $\widetilde{v}_1(t)=v_1(1-t/n)$ at $t=0$
provide the elements involved in the integral verified by 
$b_n^{>x\sigma\sqrt{n}}$ at a high order asymptotics.
Expansions of the following items of Equation (\ref{eq:u0Fgeqh2}) 
can be computed by Newton
 iteration~\cite{aecf-2017-livre}.
\begin{enumerate}
\setlength{\itemsep}{-2pt}
\item \label{it:u1v1} $\widetilde{u}_1(t)$ and $\widetilde{v}_1(t)$, \label{it:uv}
\item \label{it:m} $1/\widetilde{v}_1(t)$ and $m(t)= \dfrac{d}{d
  t}\dfrac{1}{\widetilde{v}_1(t)}=-\dfrac{\widetilde{v}'_1(t)}{v_1^2(t)}$,\label{it:mz} and  $s(t)=\widetilde{u}_1(t)\times m(t)$.
\item \label{it:log} $\log(\widetilde{u}_1(t))$ and $\log(\widetilde{v}_1(t))$,\label{it:lulv}
\item \label{it:loglog}
  $T(n,t,x)=x\sigma\sqrt{n}\Big(\log\big(\widetilde{u}_1(t)\big)-\log\big(\widetilde{v}_1(t)\big)\Big)$,
  and 
 $E(n,t,x)=\exp(T(n,t,x))$,\label{it:E}
\item \label{it:exp} ${\displaystyle
  b_n^{>x\sigma\sqrt{n}}=\dfrac{1}{2\pi i}\oint\frac{1}{(1-t/n)^{n}}
  E(n,t,x)\times s(t)dt}\qquad \left(z=1-\dfrac{t}{n},\quad t=s^2\right).$
\label{it:end}
\end{enumerate}

Inserting the expansions of items~\ref{it:u1v1} to~\ref{it:loglog} into
 Equation~\ref{eq:u0Fgeqh2}, we get at order $m$
\begin{equation}
\label{eq:bnextend}
 b_{n,luka}^{>x\sigma\sqrt{n}}=\frac{1}{2\pi i}\oint 
\frac{1}{\sigma\sqrt{2}\sqrt{n}}\frac{e^t e^{-2x\sqrt{2t}}}{\sqrt{t}}
\times\left(\sum_{k=0}^m n^{-k/2} S_k(t^{1/2},x)dt
+O\left(\frac{1}{n^{(m+1)/2}}\right)\right),
\end{equation}
 where
$S_k(s,x)$ is a multivariate polynomial of  degree $k+1$ in the
variable $s$ and  $\lfloor k/2\rfloor$ in the variable $x$ (see Section~\ref{subsub:He}).
By following the same steps as in Section~\ref{sec:semilarge} but
 at a higher asymptotic order, recalling that we are in the
 probabilistic setting 
$P(1)=1$
 with zero drift $P'(1)=0$, we have
 ${\displaystyle \beta_n^{>h}=b_n^{>x\sigma\sqrt{n}}/b_n^{<+\infty}}$
 which verifies
the following formula where $H\!e_i:=H\!e_i(4x)$, \linebreak$\ \sigma^2=P''(1),\ \xi=P'''(1)$ and $\theta=P''''(1)$
\scriptsize
\begin{align}
\nonumber
\frac{b^{>x\sigma \sqrt n}_{n,luka}}{b^{x\leq\infty}_{n,luka}}\times\frac{\sqrt{2\pi n}}{\exp(-2x^2)}=
{\it He}_{{0}}&+{\frac {{\it He}_{{1}}}{\sqrt {n}} \left(-\frac{3}{2\sigma
}-\frac {\xi}{6{\sigma}^{3}} \right) }\\\nonumber
&+\frac {1}{n}
 \left(
{\frac {{\it He}_{{4}}}{128}}+
{\it He}_{{3}} \left(-\frac{1}{12\sigma^2}-\,{\frac {\xi}{24
{\sigma}^{4}}}+{\frac {\theta}{96\,{\sigma}^{4}}}-{\frac {5\,{\xi}^{2}
}{288\,{\sigma}^{6}}}-\frac{1}{16} \right) x\right.\\
\label{eq:getorder3He}
&\qquad\qquad+\left.{\it He}_{{2}} \left(\frac{5}{4\sigma}^{2}+{\frac
   {7\,\xi}{24\,{\sigma}^{4}}}
-{\frac {\theta}{32
{\sigma}^{4}}}+{\frac {5\,{\xi}^{2}}{96\,{\sigma}^{6}}}+\frac{3}{16}
\right) \right) +O \left( {n}^{-3/2} \right)\\ 
\nonumber
=1+&\frac {x}{\sqrt {n}} \left(-\frac{6}{\sigma}\,-\,{\frac {2\xi}{{
3\sigma}^{3}}} \right) \\ \nonumber
&+\frac {1}{n}
 \left( -\frac{5}{\sigma^2}+\frac{1}{\sigma^4}\left(-\frac{7\xi}{6}+\frac{\theta}{8}\right)-\frac{5\,\xi^{2}}{24\,{\sigma}
^{6}}-\frac{3}{8}
+ \left(\frac{ 24}{\sigma^{2}}+\frac {1}{\sigma^4} \left( 
\frac {20\,\xi}{3}-\theta \right) +\frac{5\xi^2}{3\sigma^6}+3
\right) x^{2} \right.\\ \label{eq:getorder3}
&\qquad\qquad+\left.\left( -\frac{16}{3\sigma^2}+\frac{1}{3\sigma^4}\left(-8\xi+2\theta\right)-\frac {10\xi^2}{9\sigma^6}-
2 \right) x^4 \right) +O \left( {n}^{-3/2} \right). 
\end{align}
\normalsize
We computed
$B_{\operatorname{order}=7}(x,n)=b_n^{>x\sigma\sqrt{n}}/b_n^{<\infty}$ 
 at
order $n=7$ and extracted the first terms at order $n=3/2$ to provide
Equations
(\ref{eq:getorder3He},\ref{eq:getorder3}). In this last equation we
correct
the term $[n^{-1}][x^0]$ of the
equivalent formula~\footnote{Banderier-Nicodeme~\cite{BaNi2010}
considers only the first term in the expansion of $b_n^{<\infty}$.} in Banderier-Nicodeme~\cite{BaNi2010}.
The expansion can naturally be pushed to higher orders. 

\paragraph{Numerical check.}
 We use as tools of verification
the bridges with jumps $(+1,-1)$ and characteristic polynomial
\mbox{$P(u)=(u+1/u)/2$}. 

We obtain by computing at order $7$ the expansion of Equation
(\ref{eq:ucond})
\[b_n^{<\infty}\times\sqrt{2\pi n}=1-\frac{4}{n}+\frac{1}{32
  n^2}+\frac{5}{128 n^3}-\frac{21}{2048n^4}-\frac{399}{8192
  n^5}-\frac{8142861}{55296 n^6}+O\left(\frac{1}{n^7}\right).\]
 We substitute the moments
$(\sigma^2,\xi,\theta,\cdots)$ in
$B_{\operatorname{order}=7}(n,x)=b_n^{<h}/b_n^{<\infty}$ by the moments of $P(u)$ at $1$,
$\left(P^{(i)}(u)|{u=1}\right)=(1,-3,12,-60,\cdots)$. Then we
  compare
directly with the result obtained by D\'esir\'e Andr\'e reflexion 
(see Feller~\cite{Feller1950} p. 72, and Section \ref{sec:Introduct});
this reflexion principle asserts that the number of
bridges of length $m=2k$ with  height at least $h$ is equal to the
number
of walks of length $m$ terminating at height $+2h$, therefore
 the
corresponding probability is ${\displaystyle\Pr_{\text{Andr\'e}}(m,h)}=
\binom{m}{m/2+h}/\binom{m}{m/2}$. Remarking that the inequality giving
$\beta_n^{>x\sigma\sqrt{n}}$ in Equation (\ref{eq:BhF0}) is strict, with $h=9, n=64$, $x=(h-1)/\sqrt{n}=1$, 
we get $B_{\operatorname{order}=7}(64,x)-{\displaystyle\Pr_{\text{Andr\'e}}(64,9)}\approx
2\times 10^{-8}$.
See 
Figure~\ref{fig:wksheet1}
and the Maple Script\\
\url{https://lipn.univ-paris13.fr/~nicodeme/Publications/heightofbridge.mpl}.

\section{Conclusion}
\label{sec:conclusion}
We provide in this article a rigorous proof of the law 
 of the height
of discrete bridges, including the case of periodic walks, with a
convergence as expected to a Rayleigh law.
 Using the result
 of Banderier-Nicodeme~\cite{BaNi2010} we provide an algorithmic method to compute
more precise expansions of the convergence to the Rayleigh law for
{\L}ukasiewicz bridges. The detailed law of periodic walks could be
later worked out, in particular for simple walks with only one positive
and one negative jump, akin to the Duchon walk of
Figure~\ref{fig:duchoncontour}.
We propose in Section~\ref{sec:tight} a conjecture that could lead to local
refinements of the strong embedding results.

\paragraph{Acknowledgements.} 
We thank Cyril Banderier for our common
2010 article. His 2002 article with Philippe Flajolet 
provided us with rich materials that we could use in the present
article. We are grateful  to Michael Wallner who pointed the simplification
error in the 2010 article; this gave us the motivation to find a
correcting 
proof. We also  thank Alin Bostan and
Bruno Salvy for indicating  me 
the main properties of algebraic functions and the corresponding references.

\begin{figure}[ht]
\setlength{\unitlength}{1cm}
\begin{picture}(0,0)(0,16)
\put(-1.8,0){\includegraphics[width=0.7\textwidth]{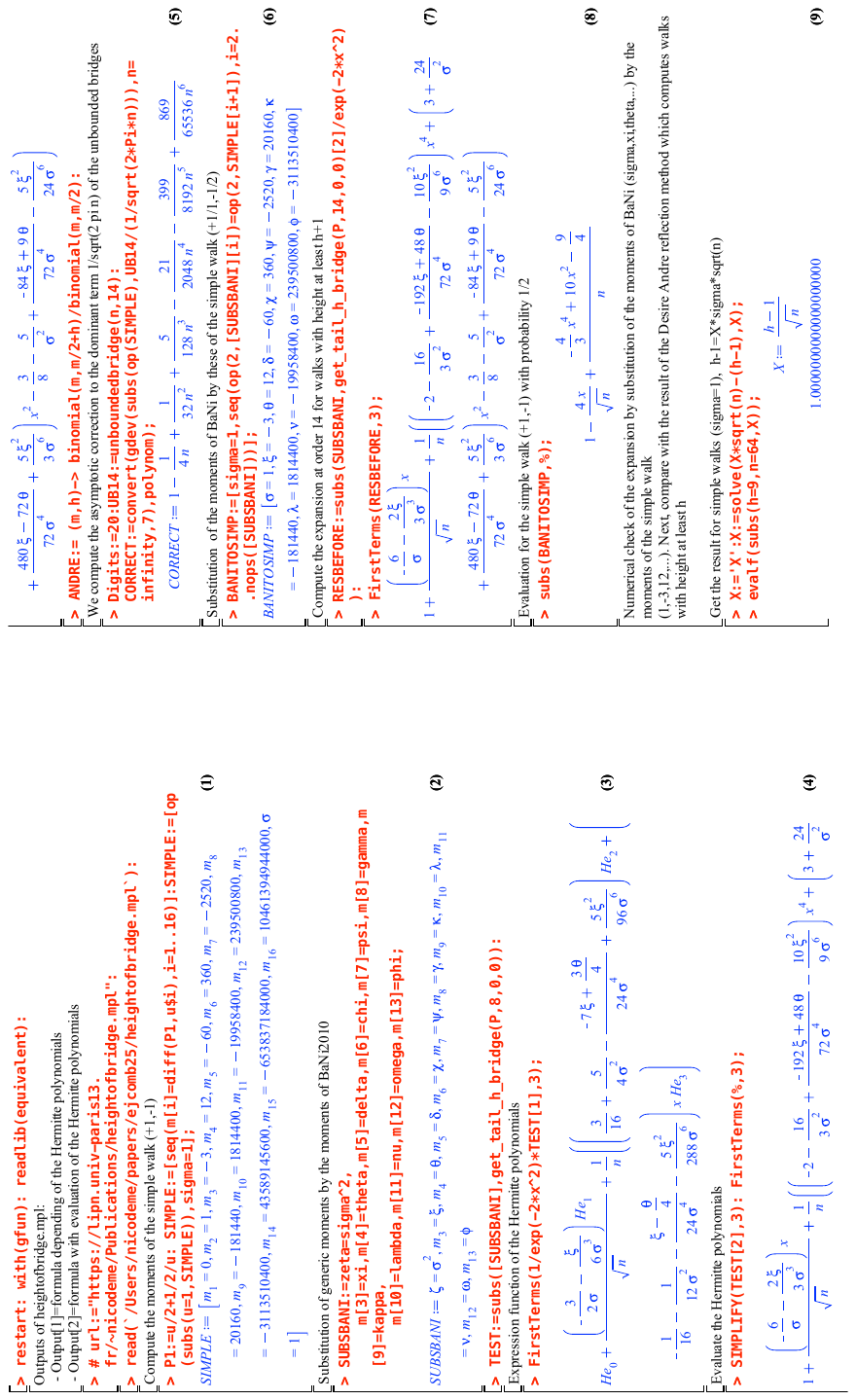}}
\put(6.8,0){\includegraphics[width=0.7\textwidth]{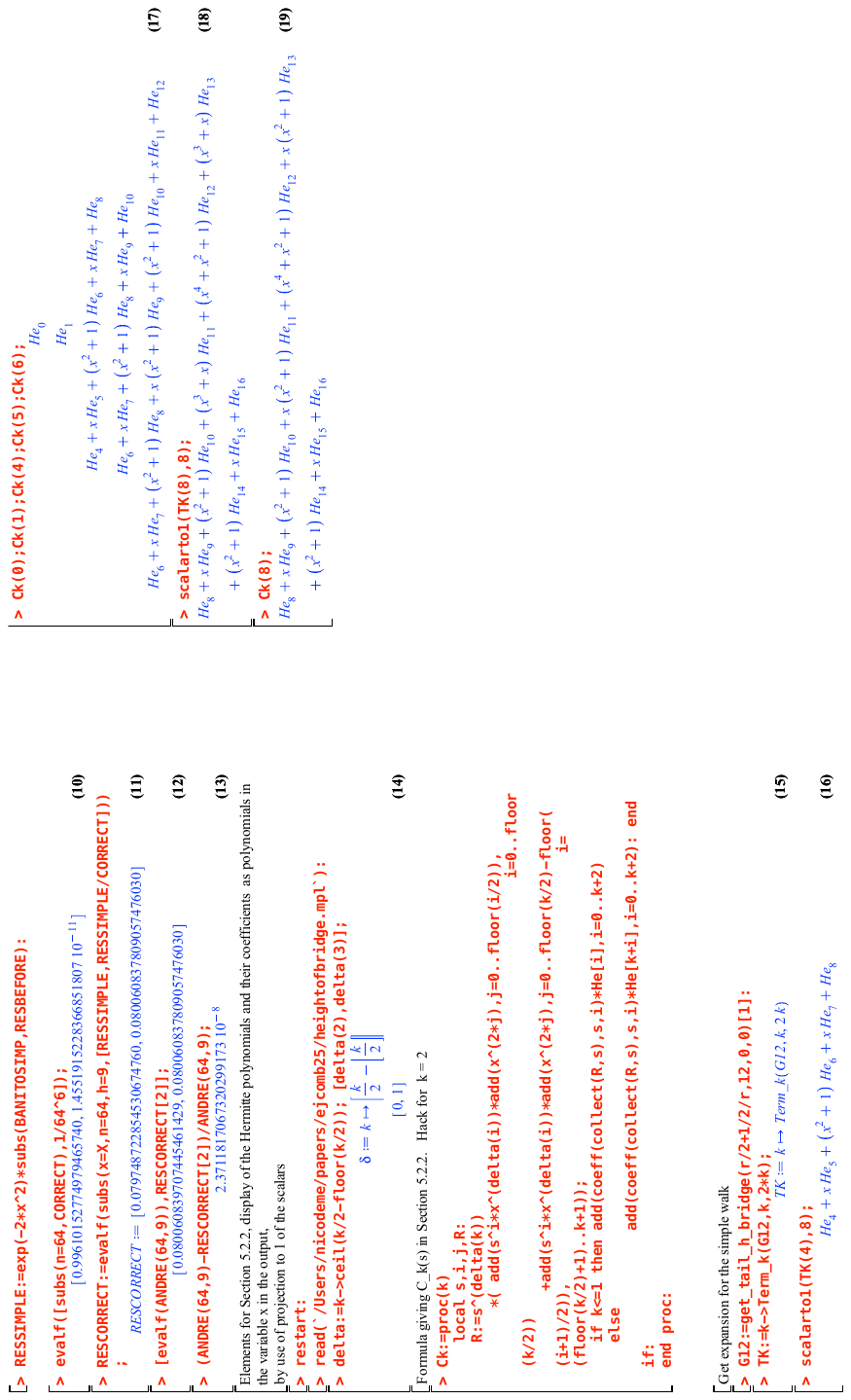}}
\end{picture}
\caption{\rule{0cm}{15.5cm}Maple worksheet}
\label{fig:wksheet1}
\end{figure}
\clearpage

\bibliographystyle{acm}
\bibliography{rhbridge}
\appendix

\section{Decomposition  of the expansion}
We are looking for an expression providing the occurrences of the polynomials
$H\!e_r(x)$ in the expansion of $b_{n,luka}^{>h}$ in Equation (\ref{eq:bnextend}).

We will use the mapping $s^r\leadsto \operatorname{He}_{r+1}$ of Equation
(\ref{eq:mapts}) to this aim.

Let us express the expansions of the terms of
 Equation~(\ref{eq:u0Fgeqh2}) with respect to
 $X=\dfrac{\sqrt{t}}{\sqrt{n}}$ at $X=0$  where we set by
 projection to $1$ the non-null
 numeric coefficients~\footnote{These coefficients are functions of
 the $k$th derivatives of $P(u)$ evaluated at $z=\rho$.} of $X^i$ for $i\geq 0$, denoting by $\stackrel{1}{=}$ these
 expansions; we observe that $\sigma\eq1 1$.
\begin{align}
\nonumber
 F=\exp(X)&=1+X+\frac{X^2}{2!}+\frac{X^3}{3!}+\dots \eq1
1+X+X^2+X^3+\dots,\\
\Longrightarrow
\label{eq:expHe}
\phantom{bbbbbb} &e^{2x^2}\times \frac{1}{2i\pi}\oint e^te^{-2x\sqrt{2t}}t^{k/2}dt \eq1
e^{x^2}\times \oint e^te^{-x\sqrt{t}}t^{k/2}dt\eq1 
\operatorname{He}^1_{k+1}(x),
\end{align}
where
\begin{align*} 
\operatorname{He}^1_k(4x)\eq1 \operatorname{He}_k(x) \quad\text{ and }\quad
 \operatorname{He}^1_{k+1}(x) \text{ is given by }
\operatorname{He}^1_{k+1}(x)=x\operatorname{He}^1_k(x)+(H\!e^1_k)'(x).
\end{align*}
We state the following lemma, where $G(X)=\dfrac{u_1(X)}{v_1(X)}$
and
$R(X)=u_1(X)\dfrac{v'_1(X)}{v_1^2(X)}$, while $z=1-\dfrac{t}{n}=1-X^2$
\begin{lemma}
\label{lem:eq1}
\begin{flalign}
\label{eq:1uv}
(a)\ & u_1(X)\eq1 v_1(X)\eq1
G(X)\eq1 R(X) \eq1 \frac{1}{1-X},\\[0ex]
\label{eq:ShX}
(b)\ & S_h(X):=e^{-x\sqrt{t}}\times G(X)^{x\sqrt{t}/X}\eq1 \sum_{i\geq
  0}S_i(W)X^{i},\qquad S_i(W)=\sum_{1\leq j \leq i}^jW^j,\ W=x\sqrt{t}
\end{flalign}
\begin{flalign}
\label{eq:eqT}
(c)\ & e^{-t}\frac{1}{z(X)^n}=e^{-t}\left(\frac{1}{1-X^2}\right)^{t/X^2}\eq1 1+
  \sum_{i\geq  1}T_i(t)X^{2i}\qquad  T_i(t)=t\sum_{0\leq j\leq i-1}t^j,\\
(d)\ & dX = \frac{1}{\sqrt{t}\sqrt{n}}dt.
\end{flalign}
\end{lemma}
\begin{proof}
Equations (\ref{eq:u1v1z}, \ref{eq:uuovvv} provide (a); we also get
\begin{equation}
S_h(X)\eq1\exp\left(\frac{W}{X}(\log(u_1(X))-\log(v_1(X))\right)\eq1
\exp\left(\frac{W}{1-X^2}\right),\qquad W=x\sqrt{t}. 
\end{equation}
Using a ``projected'' Fa{\`a} di Bruno Formula (see ~\cite{FlajoletSedgewick2009}
p.188), we have  $S_h(X)=f(g(X))=\sum_n h_n/n!$  with  $f(X)=\exp(W X)=\sum_n W^nX^n/n!$ and
$g(X)=\dfrac{1}{1-X^2}$, and therefore (b) follows from
\begin{align*}
f(g(X))&=\sum_n
h_n\frac{X^n}{n!}\eq1\sum_{k\geq0} W^k\left(\frac{1}{1-X^2}\right)^k\\
&\eq1
1+\frac{W}{1-X^2}+\frac{W^2}{(1-X^2)^2}+\dots\frac{W^k}{(1-X^2)^k}+\dots
\eq1 \sum_{i\geq 0}\sum_{1\leq j\leq i}W^j X^{2i}.
\end{align*}
Similarly, (c) follows from
\begin{align}
\label{eq:expX}
e^{-t}\frac{1}{z(X)^n}&=e^{-t}\exp\left(\!\frac{t\log(1-X^2)}{X^2}\!
\!\right)
     \eq1 1+\sum_{i\geq 1}\left(\frac{t}{1-X^2}\right)^i
\eq1 1+\sum_{i\geq 1}X^{2i} \,\,t\!\!\!\!\!\sum_{0\leq j \leq i-1}\!\!\!t^j.
\end{align}
\end{proof}
\subsection{Collecting  the terms
  $\operatorname{He}_i$ in the expansions.}
\label{subsub:He}
Writing $\Phi(X):=e^{-t}\times \dfrac{S_h(X)R(X)}{z(X)^n}$
 and $s=\sqrt{t}$ leads to
\begin{align}
\label{eq:PhiX}
&\frac{1}{1-X}\sum_{i\geq 0} a_{2i}(s)X^{2i}\eq1\sum_{i\geq 0}\sum_{0\leq
  j\leq i}a_{2j}(s)(1+X)X^{2j}\\
\nonumber
&\qquad\qquad\Longrightarrow \Phi(X)\eq1
  \sum_{r\geq 0}(1+X)X^{2r}\sum_{i=0}^r S_i(xs)T_{r-i}(s^2),
\end{align}
where $S_i(xs)=\sum_{1\leq j \leq i}(xs)^j$ and
$T_{r-i}(s^2)=s^2\sum_{0\leq j\leq r-i}s^{2j}$.
We set
\[\qquad\delta_k=\lceil k/2-\lfloor k/2 \rfloor \rceil, \text{ that verifies }
\left\{\begin{array}{l}\delta_{2i}=0\\\delta_{2i+1}=1\end{array}\right..\]
Since  $s=\sqrt{t}$ and $dX\eq1 \dfrac{1}{\sqrt{t}\sqrt{n}}dt$, by
projection of  Equation (\ref{eq:perluka}) for $\ell=0$, 
with
$C_k(s):=[X^k]\Phi(X)$ 
we obtain \footnote{See Figure~\ref{fig:wksheet1} (Right).}  
\begin{equation}
\label{eq:Xfin}
C_k(s)\eq1 s^{\delta_k}\times\left(
       \sum_{0\leq i \leq\lfloor k/2\rfloor}s^ix^{\delta_i}\!\!\sum_{0\leq j \leq \lfloor i/2\rfloor }x^{2j}
      +\sum_{\lfloor k/2\rfloor+1\leq i \leq k+1}s^ix^{\delta_i}\!\!\sum_{0\leq j \leq \lfloor k/2\rfloor
        -\lfloor (i+1)/2\rfloor} x^{2j}\right).
\end{equation}
This leads to the following proposition, where
$dX=\dfrac{1}{\sqrt{t}\sqrt{n}}dt$ provides a factor $s^{-1}$
\begin{proposition}
\begin{equation}
\label{prop:bneq1}
\frac{b_{n,luka}^{>h}\times e^{2x^2}}{\sqrt{2\pi n}} \eq1 
  \sum_{k\geq 0}
    \frac{1}{n^{k/2}}\left.\frac{C_k(s)}{s}\right|_{s^i=\operatorname{He}_{k+i-1}},
\end{equation} 
where $C_k(s)=[X^k]\dfrac{S_h(X)R(X)}{z(X)^n}$ is the $k^{th}$
term of the projection to $1$ of the integrand of Equation
(\ref{eq:u0Fgeqh2}) and is given in Equation (\ref{eq:Xfin}).

Expansion at order $2$ gives
\begin{equation}
\frac{b_n^{>h}\times e^{2x^2}}{\sqrt{2\pi n}} \eq1 
  \operatorname{He}_0+\frac{\operatorname{He}_1}{\sqrt{n}}+\frac{\operatorname{He}_2+x \operatorname{He}_3+\operatorname{He}_4}{n}+\frac{\operatorname{He}_3+x
    \operatorname{He}_4+\operatorname{He}_5}{n^{3/2}}+O\left(\frac{1}{n^2}\right).
\end{equation} 
\end{proposition}
\begin{remark}
The expansions of Equations (\ref{eq:perluka},\ref{eq:bnextend}) would
lead without doing the projection of the scalars to $1$ to an
antecedent $\overline{\Phi(X)}$ of the function $\Phi(X)$ of Equation
(\ref{eq:PhiX}) verifying
\begin{equation}
\overline{\Phi}(X)=\sum_{r\geq 0}(1+\beta_r X)\gamma_{2r}X^{2r}
        \sum_{i=0}^r\overline{S}_i(xs)\overline{T}_{r-i}(s^2),
\quad \left\{\begin{array}{l}
\overline{S}_i(sx)=\sigma_{i,1}xs+\sigma_{i,2}x^2s^2+\dots\\
\overline{T}_{r-i}(s^2)=s^2(\theta_{r-i,0}+\theta_{r-i,2}s^2+\dots)
\end{array}\right.
\end{equation}
Proving that none of the scalars
$\beta_r,\gamma_{2r},\sigma_{i,.},\theta_{r-i,.}$ is zero   is left open for
future work.
\end{remark}
\begin{remark}
\label{rem:He}
From the recurrence giving $\operatorname{He}_i$ in Equation
(\ref{eq:He}), we have $\operatorname{He}_n(x)=\Theta((4x)^n)$; on the
other side, $\dfrac{d^nP(u)}{d^nu}=\Theta((n-1)!)$.
The $kth$ term $T_k$ 
of the diverging series $S(n)$ giving $b^{>x\sigma \sqrt
  n}_n/\exp(-2x^2)$ for a {\L}ukasiewicz bridge verifies
$T_k=\Theta(4^kx^k(n-1)!/n^{k/2})$.

This suggests that the smaller term of this series is near $k=\dfrac{\sqrt{n}}{4x}$
\end{remark}

\end{document}